\documentclass[11pt,a4paper]{article}
 \usepackage{algorithm}
\usepackage{algorithmic}
\usepackage{multirow}
\usepackage{subfigure}
 \usepackage{epsfig}
 \usepackage{epstopdf}
\usepackage[T1]{fontenc}
\usepackage{geometry}
\usepackage{amsbsy,amsmath,latexsym,amsfonts, epsfig, color, authblk, amssymb, graphics, bm}
\usepackage{epsf,slidesec,epic,eepic}
\usepackage{fancybox}
\usepackage{fancyhdr}
\usepackage{setspace}
\usepackage{cases}
\usepackage{nccmath}
\usepackage{url}
\newtheorem{theorem}{Theorem}[section]
\newtheorem{lemma}{Lemma}[section]
\newtheorem{malgorithm}{Algorithm}[section]

\newtheorem{example}{Example}[section]

\newtheorem{corollary}{Corollary}[section]
\newtheorem{assumption}{Assumption}
\newtheorem{remark}{Remark}[section]

\newtheorem{alemma}{Lemma}

 \newenvironment{aproof}{{\noindent}}{\hfill$\Box$\medskip}
 \newenvironment{proof}{{\noindent \bf Proof:}}{\hfill$\Box$\medskip}

\definecolor{lred}{rgb}{1,0.8,0.8}
\definecolor{lblue}{rgb}{0.8,0.8,1}
\definecolor{dred}{rgb}{0.6,0,0}
\definecolor{dblue}{rgb}{0,0,0.5}
\definecolor{dgreen}{rgb}{0,0.5,0.5}

 \title{GEP-MSCRA for computing the group zero-norm regularized least squares estimator\footnote{Supported by the National Natural Science Foundation of China under project No.11571120 and No.11701186, and the Natural Science Foundation of Guangdong Province under project No.2015A030313214 and No.2017A030310418.}}

 \author{Shujun Bi\footnote{bishj@scut.edu.cn. School of Mathematics, South China University of Technology, Guangzhou.}
 \ \ {\rm and}\ \
 Shaohua Pan\footnote{shhpan@scut.edu.cn. School of Mathematics, South China University of Technology, Guangzhou.}}

\begin{document}

 \maketitle

 \begin{abstract}
  This paper concerns with the group zero-norm regularized least squares estimator
  which, in terms of the variational characterization of the zero-norm, can be obtained
  from a mathematical program with equilibrium constraints (MPEC). By developing
  the global exact penalty for the MPEC, this estimator is shown to arise from
  an exact penalization problem that not only has a favorable bilinear structure
  but also implies a recipe to deliver equivalent DC estimators such as the SCAD
  and MCP estimators. We propose a multi-stage convex relaxation approach
  (GEP-MSCRA) for computing this estimator, and under a restricted strong convexity
  assumption on the design matrix, establish its theoretical guarantees
  which include the decreasing of the error bounds for the iterates to the true
  coefficient vector and the coincidence of the iterates after finite steps with
  the oracle estimator. Finally, we implement the GEP-MSCRA with the subproblems
  solved by a semismooth Newton augmented Lagrangian method (ALM) and compare
  its performance with that of SLEP and MALSAR, the solvers for the weighted
  $\ell_{2,1}$-norm regularized estimator, on synthetic group sparse regression problems
  and real multi-task learning problems. Numerical comparison indicates that
  the GEP-MSCRA has significant advantage in reducing error and achieving
  better sparsity than the SLEP and the MALSAR do.
 \end{abstract}

% \medskip
% \noindent
% {\color{red} Mathematics Subject Classification(2010). 90C27, 90C33, 49M20}

%------------------------------------------------------------------------------------- Introduction
  \section{Introduction}\label{sec1}

  Many regression and learning problems aim at finding important explanatory factors
  in predicting the response variable, where each explanatory factor may be
  represented by a group of derived input variables
  (see, e.g., \cite{Evgeniou-ML,Yuan06,Huang-AOS,Meier08,Obozinski10,ZhangGY08,Ye-JMLR,Nandy17}).
  The most common example is the multifactor analysis-of-variance problem,
  in which each factor may have several levels and can be expressed through
  a group of dummy variables. Let $\mathcal{J}_1,\mathcal{J}_2,\ldots,\mathcal{J}_m$
  be a collection of index sets to represent the group structure of explanatory factors,
  where $\mathcal{J}_i\cap\mathcal{J}_j=\emptyset$ for all $i\ne j\in\{1,\ldots,m\}$
  and $\bigcup_{i=1}^m\mathcal{J}_i=\{1,2,\ldots,p\}$. This class of regression problems
  can be stated via the following observation model
  \begin{equation}\label{model}
    b ={\textstyle \sum_{i=1}^m} A_{\!_{\mathcal{J}_i}}\overline{x}_{\!_{\mathcal{J}_i}}+\varepsilon,
  \end{equation}
  where $\overline{x}\in\mathbb{R}^p$ is the true (unknown) coefficient vector,
  $A_{\!_{\mathcal{J}_i}}\,(i=1,\ldots,m)$ is an $n\times |\mathcal{J}_i|$ design matrix
  corresponding to the $i$th factor, and $\varepsilon\in\mathbb{R}^n$ is the noise vector.
  Clearly, when $\mathcal{J}_i=\{i\}$ for $i=1,2,\ldots,m$,
  \eqref{model} reduces to the common linear regression model.

  \medskip

  Sparse estimation, using penalization or regularization technique to perform
  variable selection and estimation simultaneously, has become a mainstream approach
  especially for high-dimensional data \cite{Bhlmann11}. In the past decade, some popular penalized estimators
  were proposed one after another, including the $\ell_1$-type estimators such as
  the Lasso \cite{Tib96} and the Dantzig \cite{CT07}, and the nonconvex penalized
  estimators such as the SCAD \cite{FanLi01} and the MCP \cite{Zhang10}. For the model \eqref{model},
  one may embrace the $\ell_{2,1}$-norm regularized estimator due to the simplicity of computation
  (see, e.g., \cite{Bach08,Yuan06}), but this estimator inherits the bias of the Lasso.
  The major reason for this dilemma is the significant difference between the $\ell_{1}$-norm
  and the zero-norm (or cardinality function). To enhance the quality of the $l_{1}$-type selector,
  some researchers focused on the estimator induced by nonconvex surrogates for the zero-norm
  regularized problem, such as the bridge \cite{HHM}, the SCAD \cite{FanLi01} and
  the MCP \cite{Zhang10}. In particular, some algorithms were also developed
  for computing these nonconvex penalized estimators
  \cite{FanLi01,Zou06,MFH-JASA,BH-AAS,Breheny-SC15}; for example,
  the local quadratic approximation (LQA) algorithm \cite{FanLi01} and
  the local linear approximation approximation (LLA) algorithm \cite{Zou06}.
  Recently, Fan, Xue and Zou  \cite{FanXZ14} also provided a unified theory to
  show explicitly how to obtain the oracle solution via the LLA algorithm
  for the class of folded concave penalized estimators,
  which covers the SCAD and MCP as special cases.

  \medskip

  Let $A\!:=[A_{\!_{\mathcal{J}_1}}\ A_{\!_{\mathcal{J}_2}}\cdots A_{\!_{\mathcal{J}_m}}]\in\mathbb{R}^{n\times p}$
  and $\mathcal{G}(x)\!:=\big(\|x_{\!_{\mathcal{J}_1}}\|,\|x_{\!_{\mathcal{J}_2}}\|,\ldots,\|x_{\!_{\mathcal{J}_m}}\|\big)^{\mathbb{T}}$
  for $x\in\mathbb{R}^p$. In this work, we are interested in the following group zero-norm regularized estimator
  \begin{equation}\label{G-Sparse}
   \widehat{x}\in\mathop{\arg\min}_{x\in\Omega}\left\{\frac{\nu}{2n}\|Ax-b\|^2+\|\mathcal{G}(x)\|_{0}\right\},
  \end{equation}
  where $\nu> 0$ is the regularization parameter, $\|\cdot\|_0$ denotes
  the zero-norm of a vector, and $\Omega:=\{x\in\mathbb{R}^{p}\!: \|x\|_\infty\le R\}$
  for some $R>0$. Here, the simple constraint $x\in\Omega$ is imposed to \eqref{G-Sparse}
  in order to guarantee that the group zero-norm estimator $\widehat{x}$ is well-defined.
  In fact, similar simple constraints are also used for the $\ell_1$-regularized models
  (see \cite{Agarwal12}).
  The estimator $\widehat{x}$ may be unacceptable for many statisticians since,
  compared with the convex $\ell_{2,1}$-regularized estimator and the nonconvex
  SCAD and MCP penalized estimators, it seems that $\widehat{x}$ is unapproachable
  due to the combinatorial property of the zero-norm. The main motivation for us
  to study such an estimator comes from the following facts:
  \begin{itemize}
    \item {\bf Good group sparsity and unbiasedness of $\widehat{x}$}.
          By the definition of $\|\cdot\|_0$, clearly, $\widehat{x}$ can automatically
          set small estimated coeffcients to be zero, which reduces well the model complexity.
          In addition, by following the analysis in \cite[Section 2]{FanLi01}, $\widehat{x}$
          is nearly unbiased when $A$ is orthonormal and the true coefficients are not too small.

   \item  {\bf The estimator $\widehat{x}$ is the restriction of the SCAD and MCP over
           the ball $\Omega$}. The SCAD and MCP estimators were well studied in the past ten years,
          but there are few works to discuss their relation with the zero-norm
          regularized estimator except that they are effective nonconvex surrogates
          of the latter. In Section \ref{sec2}, we shall show that the SCAD and MCP
          functions arise from the global exact penalty for the equivalent
          MPEC (mathematical program with equilibrium constraints) of \eqref{G-Sparse},
          and so $\widehat{x}$ is the restriction of the SCAD and MCP
          estimators over the ball $\Omega$.

   \item  {\bf Approachability of the estimator $\widehat{x}$}. As will be shown in Section \ref{sec3},
          with the global exact penalty for the MPEC of \eqref{G-Sparse}
          which is actually a primal-dual equivalent model of \eqref{G-Sparse},
          there is a large space to design efficient algorithms for computing $\widehat{x}$.
  \end{itemize}

  Specifically, by means of the variational characterization of the zero-norm,
  the group zero-norm regularized problem \eqref{G-Sparse} can be rewritten as an MPEC.
  We show that the penalty problem, yielded by moving the equilibrium constraint
  into the objective, is a global exact penalty for the MPEC in the sense that
  it has the same global optimal solution set as the MPEC does. Consequently,
  one may approach the estimator $\widehat{x}$ by solving a global exact penalization problem.
  This result is significant since, on one hand, the global exact penalty is an Lipschitz
  continuous optimization problem whose objective function possesses a structure favorable
  to the design of effective algorithms; and on the other hand, it provides a recipe
  to deliver equivalent DC (difference of convex functions) penalized functions whose
  global minimizers provide an estimator with three desirable properties stated in \cite{FanLi01};
  for example, the popular SCAD and MCP penalized estimators.

  \medskip

  By the biconvex structure of the global exact penalty,
  we solve it in an alternating way and develop a multi-stage convex
  relaxation approach called GEP-MSCRA for computing $\widehat{x}$
  (see Section \ref{sec3}). The GEP-MSCRA consists of solving a sequence of
  weighted $\ell_{2,1}$-norm regularized subproblems. In this sense, it is similar
  to the LLA algorithm \cite{Zou06} for the nonconcave penalized likelihood model.
  However, it is worth emphasizing that the start-up of the LLA algorithm
  depends implicitly on an initial estimator $x^0$, while the start-up of
  the GEP-MSCRA depends explicitly on a dual variable $w^0$. In addition,
  the involving subproblems may be different since the subproblems of the LLA
  are obtained from the primal angle, and those of the GEP-MSCRA are yielded
  from the primal-dual angle. For the proposed GEP-MSCRA, under a restricted strong convex
  (RSC) assumption on $A$, we verify in Section \ref{sec4} that the error bounds
  of the iterates to the true $\overline{x}$ is decreasing as the number of
  iterates increases, and if the smallest nonzero group element of $\overline{x}$
  is not too small, the iterates after finite steps will coincide with
  the oracle solution, and hence the support of $\overline{x}$ is exactly identified.
  Since the RSC assumption holds with high probability by \cite{Negahban12},
  the GEP-MSCRA has the theoretical guarantees in a statistical sense.

  \medskip

  We implement the GEP-MSCRA by solving the weighted $\ell_{2,1}$-norm regularized
  subproblems with a semismooth Newton ALM. The semismooth Newton ALM is a dual
  method that can solve the weighted $\ell_{2,1}$-norm regularized problems
  more efficiently than the existing first-order methods by exploiting
  the second-order information of the objective function in an economic way.
  As illustrated in \cite[Section 3.3]{LiSunToh16}, the dual structure of
  the weighted $\ell_{2,1}$-norm regularized problems implies a nonsingular
  generalized Hessian matrix, which is well suitable for the semismooth
  Newton method. We compare the performance of the GEP-MSCRA with that of
  the SLEP and the MALSAR in \cite{LiuYe09}, the solvers to
  the unconstrained weighted $\ell_{2,1}$-norm regularized least squares problems
  on synthetic group sparse regression problems and real multi-task learning problems,
  respectively. Numerical comparisons demonstrates that the GEP-MSCRA has
  a remarkable advantage in reducing error and achieving the exact group sparsity
  although it requires a little more time than the SLEP and the MALSAR do;
  for example, for the synthetic group sparse regression problems,
  the GEP-MSCRA reduces the relative recovery error of the SLEP at least
  $60\%$ for the design matrix of Gaussian or sub-Gaussian type
  (see the first four subfigures in Figure \ref{fig3}),
  and for the real (School data) multi-task learning, the GEP-MSCRA can reduce
  the prediction error of the MALSAR at least $20\%$ when there are more than
  $50\%$ examples are used as training samples (see Figure \ref{fig8}).

  \medskip

 To close this section, we introduce some necessary notations. We denote $\|A\|$
 by the spectral norm and $\|A\|_{2,\infty}$ by the maximum column norm of $A$,
 respectively. Let $e$ and $I$ denote the vector of all ones and the identity
 matrix, respectively, whose dimensions are known from the context.
 For a convex function $g\!:\mathbb{R}\to(-\infty,+\infty]$,
 $g^*$ denotes the conjugate of $g$; for a given closed set $S\subseteq\mathbb{R}^n$,
 $\delta_{S}(\cdot)$ means the indicator function over the set $S$, i.e., $\delta_{S}(x)=0$
 if $x\in S$ and otherwise $\delta_{S}(x)=+\infty$; and for a given index set
 $F\subseteq\{1,\ldots,m\}$, write $F^{c}=\{1,\ldots,m\}\backslash F$
 and $\mathcal{J}_F:=\bigcup_{i\in F}\mathcal{J}_i$, and denote by
 $\mathbb{I}_{F}(\cdot)$ the characteristic function of $F$, i.e.,
 $\mathbb{I}_{F}(i)=1$ if $i\in F$ and otherwise $\mathbb{I}_{F}(i)=0$.

%---------------------------------------------------------------------------- Section 2
 \section{A new perspective on the estimator $\widehat{x}$}\label{sec2}

 We shall examine the estimator $\widehat{x}$ from an equivalent MPEC of \eqref{G-Sparse}
 and a global exact penalty of this MPEC, and conclude that $\widehat{x}$
 can be obtained by solving an exact penalty problem which is constructed
 by moving the complementarity (or equilibrium) constraint into
 the objective of the MPEC. For convenience, we write $f(x)\!:=\frac{1}{2n}\|Ax-b\|^2$
 and denote by $L_{\!f}$ the Lipschitz constant of $f$ relative to the set $\Omega$.
 One will see that the results of this section are also applicable to a general
 continuous loss function.

 \medskip

 Let $\Phi$ denote the family of closed proper convex functions
 $\phi\!:\mathbb{R}\to(-\infty,+\infty]$ satisfying $[0,1]\subseteq {\rm int}({\rm dom}\phi)$,
 $\phi(1)=1$ and $\phi(t_{\phi}^*)=0$ where $t_{\phi}^*$ is the unique minimizer
 of $\phi$ over $[0,1]$. Let $\overline{t}_{\phi}$ be the minimum element in
 $[t_{\phi}^*,1)$ such that $\frac{1}{1-t_{\phi}^*}\in\partial\phi(\overline{t}_{\phi})$,
 where $\partial\phi$ is the subdifferential mapping of $\phi$.
 The existence of such $\overline{t}_{\phi}$ is guaranteed by Lemma \ref{lemma1-phi}.

 \medskip

 Now we show that with an arbitrary $\phi\in\Phi$,
 the problem \eqref{G-Sparse} can be rewritten as an MPEC.
 Fix an arbitrary $z\in \mathbb{R}^m$ and $\phi\in\Phi$.
 By the definition of $\Phi$, one may check that
 \begin{equation}\label{vc-zn}
   \|z\|_0=\min_{w\in\mathbb{R}^m}\!\Big\{{\textstyle\sum_{i=1}^m}\phi(w_i)\!: \|z\|_1-\langle w,|z|\rangle=0,\ 0\le w\le e\Big\}.
 \end{equation}
 This variational characterization of $\|\cdot\|_0$ means that the problem \eqref{G-Sparse}
 is equivalent to
 \begin{equation}\label{vc-GS}
   \min_{x\in\Omega,w\in\mathbb{R}^m}\!\bigg\{\nu\!f(x)+\sum_{i=1}^m\phi(w_i)\!:\,
   0\le w\le e,\,\langle e-w,\mathcal{G}(x)\rangle=0\bigg\}
 \end{equation}
 in the following sense: if $x^*$ is globally optimal to \eqref{G-Sparse},
  then $(x^*,\max({\rm sign}(\|\mathcal{G}(x^*)\|),t_{\phi}^*e))$ is
  a global optimal solution of \eqref{vc-GS}; and conversely,
  if $(x^*,w^*)$ is a global optimal solution of \eqref{vc-GS},
  then $x^*$ is globally optimal to \eqref{G-Sparse} with
  $\|\mathcal{G}(x^*)\|_{0}=\sum_{i=1}^m\phi(w^*_i)$. This means that
  the difficulty to compute the estimator $\widehat{x}$ comes from
  the following equilibrium condition
  \begin{equation}\label{equilibrium}
    e-w\ge 0,\ \mathcal{G}(x)\ge 0\ \ {\rm and}\ \
    \langle e\!-\!w,\mathcal{G}(x)\rangle=0.
  \end{equation}
  Also, it is the equilibrium constraint to bring the bothersome nonconvexity
  of \eqref{G-Sparse}. Since the constraint set of \eqref{vc-GS} involves
  the equilibrium constraint \eqref{equilibrium}, we call it an MPEC.

  \medskip

  It is well known that the MPEC is a class of very difficult problems in optimization.
  In the past two decades, there was active research on its theory and algorithms
  especially for the one over the polyhedral cone, and the interested readers may refer to
  \cite{LuoPang96,YeZhu97}. We notice that most of existing algorithms are generic
  and inappropriate for solving \eqref{vc-GS}. To handle the tough equilibrium constraint,
  we here consider its penalized version
  \begin{equation}\label{pvc-GS}
   \min_{x\in\Omega,w\in[0,e]}\!\bigg\{\nu f(x)+\sum_{i=1}^m\phi(w_i) +\rho\langle e\!-\!w,\mathcal{G}(x)\rangle\bigg\}
  \end{equation}
  where $\rho>0$ is the penalty factor. The following theorem states that
  \eqref{pvc-GS} is a global exact penalty for \eqref{vc-GS} in the
  sense that it has the same global optimal solution set as \eqref{vc-GS} does.
 %---------------------------------------------------------------------------------------Theorem 3.1
  \begin{theorem}\label{exact-penalty}
   Let $\phi\in\Phi$. Then, for every $\rho>\overline{\rho}=\nu L_{\!f}\frac{(1-t_{\phi}^*)\phi_-'(1)}{1-\overline{t}_{\phi}}$,
   we have $\mathcal{S}_{\rho}^*=\mathcal{S}^*$, where $\mathcal{S}_{\rho}^*$
   is the global optimal solution set of \eqref{pvc-GS} associated to $\rho$,
   and $\mathcal{S}^*$ is that of \eqref{vc-GS}.
  \end{theorem}

  The proof of Theorem \ref{exact-penalty} is included in Appendix B.
  From the proof, we see that the constraint $x\in\Omega$ in \eqref{G-Sparse}
  is also necessary to establish the exact penalty for the MPEC.
  By Theorem \ref{exact-penalty} and the equivalence
  between \eqref{vc-GS} and \eqref{G-Sparse}, the estimator $\widehat{x}$ can be computed by
  solving a single penalty problem \eqref{pvc-GS} associated to $\rho>\overline{\rho}$.
  Since $[0,1]\subseteq{\rm int}({\rm dom}\phi)$, the function $\sum_{i=1}^m\phi(w_i)$
  is Lipschitzian relative to $[0,e]$ by \cite[Theorem 10.4]{Roc70},
  and so is the objective function of \eqref{pvc-GS} relative to its feasible set
  $\Omega\times[0,e]$. Thus, compared with the discontinuous nonconvex problem \eqref{G-Sparse},
  the problem \eqref{pvc-GS} is at least an Lipschitz-type one, for which the Clarke
  generalized differential \cite{Clarke83} can be used for its analysis.

  \medskip

  Interestingly, the equivalence between \eqref{G-Sparse} and \eqref{pvc-GS}
  also implies a mechanism to yield equivalent DC surrogates for the group zero-norm
  $\|\mathcal{G}(x)\|_0$, and the popular SCAD function \cite{FanLi01} and
  MCP function \cite{Zhang10} are one of the products. Next we demonstrate this fact. For each $\phi\in\Phi$,
  let $\psi\!:\mathbb{R}\to(-\infty,+\infty]$ be the associated closed proper convex function:
 \begin{equation}\label{phi-psi}
  \psi(t):=\!\left\{\!\begin{array}{cl}
                  \phi(t) &\textrm{if}\ t\in [0,1],\\
                   +\infty & \textrm{otherwise}.
                 \end{array}\right.
 \end{equation}
  By using the function $\psi$, the problem \eqref{pvc-GS} can be rewritten in the following compact form
  \begin{equation*}
   \min_{x\in\Omega,w\in\mathbb{R}^m}
   \bigg\{\nu\!f(x)+\rho\|x\|_{2,1}+\sum_{i=1}^m\big[\psi(w_i)
            -\rho w_i\|x_{\!_{\mathcal {J}_i}}\|\big]\bigg\}.
  \end{equation*}
  Together with the definition of conjugate functions and the above discussion,
  we have
 \begin{equation}\label{surrogate}
  \widehat{x}\in\mathop{\arg\min}_{x\in\Omega}
  \bigg\{f(x)+\frac{1}{\nu}\big[\rho\|x\|_{2,1}-\sum_{i=1}^m\psi^*(\rho\|x_{\!_{\mathcal {J}_i}}\|)\big]\bigg\}
  \ \ {\rm for}\ \rho>\overline{\rho}.
 \end{equation}
 This means that the following function provides an equivalent DC surrogate for $\frac{1}{\nu}\|\mathcal{G}(x)\|_0$:
 \begin{equation}\label{surrogate-fun}
  \Theta(x):=\frac{1}{\nu}{\textstyle\sum_{i=1}^m}\theta(\rho\|x_{\!_{\mathcal {J}_i}}\|)\ \ {\rm with}\ \
  \theta(s):=s-\psi^*(s).
 \end{equation}
 In particular, when $\phi$ is chosen as the one in Example \ref{SCAD},
 it becomes the SCAD function. Indeed, from the expression of $\psi^*$
 in Example \ref{SCAD} below, it follows that
 \[
   \varphi(1)\theta(\tau/{\varphi(1)})=\left\{\!\begin{array}{cl}
                      \tau & \textrm{if}\ |\tau|\leq 1,\\
                       \frac{-\tau^2+2a\tau-1}{2(a-1)}& \textrm{if}\ 1<|\tau|\le a,\\
                      \frac{a+1}{2} & \textrm{if}\ |\tau|>a
                \end{array}\right.
  \]
  which, by setting $s:={\tau}/{\lambda}$ for some constant $\lambda>0$, implies that
  \[
   \lambda^2\varphi(1)\theta({s}/{\lambda\varphi(1)})=\left\{\!\begin{array}{cl}
                      s\lambda & \textrm{if}\ |s|\leq \lambda,\\
                       \frac{-s^2+2as\lambda-\lambda^2}{2(a-1)}& \textrm{if}\ \lambda<|s|\le a\lambda,\\
                      (a+1)\lambda^2/2& \textrm{if}\ |s|>a\lambda.
                \end{array}\right.
  \]
  Thus, when $\mathcal{J}_i=\{i\}$, by taking $\nu=\frac{1}{\lambda^2\varphi(1)}$
  and $\rho=\frac{1}{\lambda\varphi(1)}$, the function $\Theta$ in \eqref{surrogate-fun}
  reduces to the SCAD function in \cite{FanLi01}. Similarly,
  when $\phi$ is chosen as the one in Example \ref{MCP},
  by taking $\nu=\frac{2}{\lambda^2 a}$ and $\rho=\frac{1}{\lambda}$,
  the function $\Theta$ in \eqref{surrogate-fun} becomes the MCP function in \cite{Zhang10}.

  \medskip

  Now we give four examples for $\phi\in\Phi$. In the sequel we shall call
  $\phi_1$-$\phi_4$ as the function in Example \ref{SCAD}-\ref{L12}, respectively,
  and $\psi_1$-$\psi_4$ as the corresponding $\psi$ defined by \eqref{phi-psi}.
%%-------------------------------------------------------------------------------------------------Example 4.2
  \begin{example}\label{SCAD}
   Take $\phi(t):=\frac{\varphi(t)}{\varphi(1)}$ with
   $\varphi(t):=\frac{a-1}{2}t^2+t$ for $t\in\mathbb{R}$, where $a>1$ is a constant.
   Clearly, $\phi\in\Phi$ with $t_{\phi}^*=0$ and $\overline{t}_{\phi}=\frac{1}{2}$.
   After a simple computation,
   \[
    \psi^*(s)=\left\{\!\begin{array}{cl}
                      0 & \textrm{if}\ |s|\leq \frac{1}{\varphi(1)},\\
                       \frac{(\varphi(1)|s|-1)^2}{2(a-1)\varphi(1)}& \textrm{if}\ \frac{1}{\varphi(1)}<|s|\le \frac{a}{\varphi(1)},\\
                      |s|-\frac{a+1}{2\varphi(1)} & \textrm{if}\ |s|>\frac{a}{\varphi(1)}.
                \end{array}\right.
   \]
  %If $\mathcal{J}_i=\{i\}$, the function $\Theta$ in \eqref{surrogate-fun}
%  with such $\psi^*$ is the SCAD function proposed in \cite{FanLi01}.
 \end{example}
 \begin{example}\label{MCP}
  Let $\phi(t)\!:=\frac{\varphi(t)}{\varphi(1)}$ with $\varphi(t):=\frac{a^2}{4}t^2-\frac{a^2}{2}t+at+\frac{(a-2)_+^2}{4}$ where $a>0$ is a constant and $(\cdot)_+=\max(0,\cdot)$. Clearly, $\phi\in\Phi$
  with $t_{\phi}^*=\frac{(a-2)_+}{a}$ and $\overline{t}_{\phi}=\max(\frac{a-1}{a},\frac{1}{2})$.
  An elementary calculation yields that $\psi^*$ takes the following form
   \[
    \psi^*(s)=\left\{\!\begin{array}{cl}
                      \frac{(a-2)_+^2}{4} & \textrm{if}\ |s|\leq \frac{a-a^2/2}{\varphi(1)},\\
                       \frac{1}{a^2\varphi(1)}\big(\frac{a^2-2a}{2}+s\varphi(1)\big)^2-\frac{(a-2)_+^2}{4\varphi(1)}& \textrm{if}\ \frac{a-a^2/2}{\varphi(1)}<|s|\le \frac{a}{\varphi(1)},\\
                      |s|-1 & \textrm{if}\ |s|>\frac{a}{\varphi(1)}.
                \end{array}\right.
   \]
  %If $\mathcal{J}_i=\{i\}$ and $a\geq2$, the function $\Theta$ in \eqref{surrogate-fun}
%  with such $\psi^*$ is the MCP function in \cite{Zhang10}.
  \end{example}
%------------------------------------------------------------------------------------------------------Example4.1
 \begin{example}\label{HARD}
  Take $\phi(t):=t$ for $t\in\mathbb{R}$. Clearly, $\phi\in\Phi$
  with $t_{\phi}^*=0$ and $\overline{t}_{\phi}=0$. Also,
  \[
    \psi^*(s)=\left\{\!\begin{array}{cl}
                     s-1 & \textrm{if}\ s>1,\\
                     0  & \textrm{if}\ s\le 1.
                \end{array}\right.
  \]
  In this case, the function $\Theta$ in \eqref{surrogate-fun} is exactly
  the capped $l_1$-surrogate of $\|\mathcal{G}(x)\|_{0}$ in \cite{Ye-JMLR}.
 \end{example}
 \begin{example}\label{L12}
  Let $\phi(t):=\frac{\varphi(t)}{\varphi(1)}$ with
  $\varphi(t):=-t-\frac{q-1}{q}(1-t+\epsilon)^{\frac{q}{q-1}}+\epsilon+\frac{q-1}{q}\,(0<q<1)$
  for $t\in (-\infty, 1+\epsilon]$, where $\epsilon\in(0,0.1)$ is a small constant.
  It is not hard to check that $\phi\in\Phi$ with $t_{\phi}^*=\epsilon$
  and $\overline{t}_{\phi}=1+\epsilon-(\frac{1-\epsilon}{1-\epsilon+\varphi(1)})^{1-q}$.
  Also, $\psi^*(s)=\frac{h(\varphi(1)s)}{\varphi(1)}$ with
  \[
   h(t):=\left\{\begin{array}{cl}
                      t+\frac{q-1}{q}\epsilon^{\frac{q}{q-1}}-\epsilon-\frac{q-1}{q} & \textrm{if}\ t>\epsilon^{\frac{1}{q-1}}-1,\\
                      (1\!+\epsilon)t-\frac{1}{q}(t+1)^q+\frac{1}{q} & \textrm{if}\ (1\!+\epsilon)^{\frac{1}{q-1}}\!-\!1<t\le\!\epsilon^{\frac{1}{q-1}}\!-\!1,\\
                      \frac{q-1}{q}(1+\epsilon)^{\frac{q}{q-1}}-\epsilon-\frac{q-1}{q} & \textrm{if}\ t\le (1\!+\epsilon)^{\frac{1}{q-1}}-1.
                \end{array}\right.
 \]
 \end{example}
% \begin{example}\label{example4}
%  Let $\phi_4(t):=\frac{\varphi(t)}{\varphi(1)}$ with
%  $\varphi(t)=-t-\ln(1-t+\epsilon)+\epsilon$ for $t\in (-\infty,1+\epsilon)$,
%  where $\epsilon\in(0,0.1)$ is a small constant. Now we have $\phi_4\in\Phi$
%  with $t_{\phi}^*=\epsilon$ and $\overline{t}_{\phi}=\epsilon$, and
%  \[
%   \psi_4^*(s)=\frac{h(\varphi(1)s)}{\varphi(1)}
%   \ \ {\rm with}\ h(t):=\left\{\begin{array}{cl}
%                      t+1+\ln(\epsilon)-\epsilon & \textrm{if}\ t>\frac{1}{\epsilon}-1,\\
%                      t(1+\epsilon)-\ln(t+1) & \textrm{if}\ \frac{1}{1+\epsilon}\!-\!1<t\le\frac{1}{\epsilon}\!-\!1,\\
%                      \ln(1+\epsilon)-\epsilon & \textrm{if}\ t\le\frac{1}{1+\epsilon}-1.
%                \end{array}\right.
% \]
% \end{example}

 To close this section, we take a look at the local optimality relation
 of \eqref{pvc-GS} and \eqref{G-Sparse}.
%------------------------------------------------------------------------------------------
 \begin{theorem}\label{local-theorem}
  If $(\overline{x},\overline{w})$ with $\langle e-\overline{w},\mathcal{G}(\overline{x})\rangle=0$
  is a local optimal solution of \eqref{pvc-GS} associated to $\rho>0$,
  then $(\overline{x},\overline{w})$ is locally optimal to \eqref{vc-GS}
  and so is $\overline{x}$ to \eqref{G-Sparse}. If $\overline{x}$ is locally
  optimal to \eqref{G-Sparse}, then $(\overline{x},\overline{w})$ with
  $\overline{w}=\max({\rm sign}(\|\mathcal{G}(\overline{x})\|),t_{\phi}^*e)$
  is locally optimal to \eqref{pvc-GS} for any $\rho>0$.
 \end{theorem}
 \begin{proof}
  Since $(\overline{x},\overline{w})$ is a local optimal solution of \eqref{pvc-GS}
  associated to $\rho>0$, there exists $\delta'>0$ such that for all
  $(x,w)\in\Omega\times[0,e]$ with $\|(x,w)-(\overline{x},\overline{w})\|\le\delta'$,
  \[
    \nu f(\overline{x})+{\textstyle\sum_{i=1}^m\phi(\overline{w}_i)}+\rho\langle e-\overline{w},\mathcal{G}(\overline{x})\rangle
    \le\nu f(x)+{\textstyle\sum_{i=1}^m\phi(w_i)}+\rho\langle e-w,\mathcal{G}(x)\rangle.
  \]
  Fix an arbitrary $(x,w)\in\mathcal{F}$ with $\|(x,w)-(\overline{x},\overline{w})\|\le\delta'/2$
  where $\mathcal{F}$ denotes the feasible set of \eqref{vc-GS}.
  Then, from the last inequality, it immediately follows that
  \begin{align*}
   &\nu f(\overline{x})+{\textstyle\sum_{i=1}^m\phi(\overline{w}_i)}
    =\nu f(\overline{x})+{\textstyle\sum_{i=1}^m\phi(\overline{w}_i)}+\rho\langle e-\overline{w},\mathcal{G}(\overline{x})\rangle\\
   &\le\nu f(x)+{\textstyle\sum_{i=1}^m\phi(w_i)}+\rho\langle e-w,\mathcal{G}(x)\rangle
   =\nu f(x)+{\textstyle\sum_{i=1}^m\phi(w_i)}.
  \end{align*}
  This shows that $(\overline{x},\overline{w})$ is a locally optimal solution of
  the problem \eqref{vc-GS}.

  \medskip

  We next argue that $\overline{x}$ is locally optimal to \eqref{G-Sparse}.
  Since $(\overline{x},\overline{w})$ is a locally optimal solution of \eqref{vc-GS},
  there exists $\widehat{\delta}>0$ such that for all
  $(x,w)\in\mathcal{F}$ with $\|(x,w)-(\overline{x},\overline{w})\|\le\widehat{\delta}$,
  \begin{equation}\label{temp-ineq}
    \nu f(\overline{x})+{\textstyle\sum_{i=1}^m\phi(\overline{w}_i)}
    \le \nu f(x)+{\textstyle\sum_{i=1}^m\phi(w_i)}.
  \end{equation}
  Let $\delta:=\min(\frac{1}{2\nu L_{\!f}},\frac{c}{4},\widehat{\delta})$
  where $c$ is the minimal nonzero component of $\mathcal {G}(\overline{x})$.
  Fix an arbitrary $x\in \Omega\cap\{z\in \mathbb{R}^p:\,\|z-\overline{x}\|\leq\delta\}$.
  We shall establish the following inequality
 \begin{equation}\label{lemmacover}
  \nu f(\overline{x})+\|\mathcal{G}(\overline{x})\|_{0}
  \le \nu f(x)+\|\mathcal{G}(x)\|_{0},
 \end{equation}
 and consequently $\overline{x}$ is a local optimal solution to \eqref{G-Sparse}.
 If $\|\mathcal{G}(x)\|_{0}\geq \|\mathcal{G}(\overline{x})\|_{0}+1$,
 \[
  \|\mathcal{G}(x)\|_{0}-\|\mathcal{G}(\overline{x})\|_{0}\geq 1>\nu L_f\delta\geq \nu L_f\|x-\overline{x}\|\ge\nu f(\overline{x})-\nu f(x).
 \]
 This implies that inequality \eqref{lemmacover} holds.
 If $\|\mathcal{G}(x)\|_{0}\leq \|\mathcal{G}(\overline{x})\|_{0}$,
 by using $\|x-\overline{x}\|\leq\frac{c}{4}$, we have
 ${\rm supp}(\mathcal{G}(x))={\rm supp}(\mathcal{G}(\overline{x}))$,
 which implies that $\langle e-\overline{w},\mathcal{G}(x)\rangle=0$.
 Thus, $(x,\overline{w})\in\mathcal{F}$ and $\|(x,\overline{w})-(\overline{x},\overline{w})\|\le \widehat{\delta}$.
 From \eqref{temp-ineq}, we obtain
 \(
   \nu f(\overline{x})+{\textstyle\sum_{i=1}^m\phi(\overline{w}_i)}
    \le \nu f(x)+{\textstyle\sum_{i=1}^m\phi(\overline{w}_i)}
 \)
 or
 \(
    \nu f(\overline{x})\le \nu f(x).
 \)
 Along with ${\rm supp}(\mathcal{G}(x))={\rm supp}(\mathcal{G}(\overline{x}))$,
 the inequality \eqref{lemmacover} follows.

 \medskip

  Since $\overline{x}$ is locally optimal to the problem \eqref{G-Sparse},
  there exists $\overline{\delta}>0$ such that for all
  $x\in\Omega$ with $\|x-\overline{x}\|\le\overline{\delta}$,
  it holds that
  \(
    \nu f(\overline{x})+\|\mathcal{G}(\overline{x})\|_0
    \le\nu f(x)+\|\mathcal{G}(x)\|_0.
  \)
  Fix an arbitrary $(x,w)\in\Omega\times[0,e]$ with
  $\|(x,w)-(\overline{x},\overline{w})\|\le\overline{\delta}$.
  Then, for any $\rho>0$,
  \begin{align*}
    \nu f(\overline{x})+{\textstyle\sum_{i=1}^m}\phi(\overline{w}_i)
    +\rho\langle e-\overline{w},\mathcal{G}(\overline{x})\rangle
    &=\nu f(\overline{x})+\|\mathcal{G}(\overline{x})\|_0
    \le\nu f(x)+\|\mathcal{G}(x)\|_0\\
    &\le\nu f(x)+{\textstyle\sum_{i=1}^m}\phi(w_i)\\
    &\le \nu f(x)+{\textstyle\sum_{i=1}^m}\phi(w_i)+\rho\langle e-w,\mathcal{G}(x)\rangle.
  \end{align*}
  where the second inequality is due to \eqref{vc-zn}.
  Thus, $(\overline{x},\overline{w})$ is locally optimal to \eqref{pvc-GS}.
 \end{proof}
%---------------------------------------------------------------------------- Section 3
  \section{GEP-MSCRA for computing the estimator $\widehat{x}$}\label{sec3}

  From the last section, to compute the estimator $\widehat{x}$,
  one only needs to solve the penalty problem \eqref{pvc-GS} associated to
  $\rho>\overline{\rho}$ with $\phi\in\Phi$, where the threshold $\overline{\rho}>0$
  is easily estimated once $\nu>0$ is given since
  $L_{\!f}=\max_{x\in\Omega}\frac{1}{n}\|A^{\mathbb{T}}(Ax-b)\|
  \le \frac{R\sqrt{p}}{n}\|A\|^2+\frac{1}{n}\|A^{\mathbb{T}}b\|$.
  For a given $\rho>\overline{\rho}$, although the problem \eqref{pvc-GS}
  is nonconvex due to the coupled term $\sum_{i=1}^m w_i\|x_{\!_{\mathcal {J}_i}}\|$,
  its special structure makes it much easier to cope with than do the problem \eqref{G-Sparse}.
  Specifically, when the variable $w$ is fixed, the problem \eqref{pvc-GS} reduces to
  a convex minimization in $x$; and when the variable $x$ is fixed, it reduces to
  a convex minimization in $w$ which, as will be shown below, has a closed-form solution.
  Motivated by this, we propose a multi-stage convex relaxation approach for
  computing $\widehat{x}$ by solving \eqref{pvc-GS} in an alternating way.

  \medskip
	
  \setlength{\fboxrule}{0.7pt}
  \noindent
  \fbox{\parbox{0.96\textwidth}
  {
 \begin{malgorithm} \label{Alg}({\bf GEP-MSCRA for computing $\widehat{x}$})
 \begin{description}
 \item[(S.0)] Choose $\phi\!\in\!\Phi$, $\nu\!>\!0$ and an initial $w^0\!\in\![0,\overline{t}_{\phi}e]$.
              Set $\lambda^0\!=\!\nu^{-1}$ and $k:=1$.

 \item[(S.1)] Compute the following minimization problem
              \begin{equation}\label{expm-subx}
               x^k\in\mathop{\arg\min}_{x\in\Omega}\bigg\{\frac{1}{2n}\|Ax-b\|^2+\lambda^{k-1}\sum_{i=1}^m(1-w_i^{k-1})\|x_{\!_{\mathcal{J}_i}}\|\bigg\}.
             \end{equation}
             \hspace*{0.1cm} If $k=1$, by the information of $x^1$ select a suitable $\rho>\overline{\rho}$
             and set $\lambda^k=\rho\nu^{-1}$.

 \item[(S.2)] Seek an optimal solution $w_i^k\ (i=1,2,\ldots,m)$ to the minimization problem
              \begin{equation}\label{expm-subw}
               w_i^k\in \mathop{\rm arg\min}_{w_i\in[0,1]}
               \left\{\phi(w_i)-\rho w_i\|x^k_{\!_{\mathcal {J}_i}}\|\right\}.
             \end{equation}

 \item[(S.3)] Set $k\leftarrow k+1$, and then go to Step (S.1).
 \end{description}
 \end{malgorithm}
 }
 }

 \bigskip

 \begin{remark}\label{remark-alg}
 {\bf(a)} By the definition of $\psi$, clearly, $w_i^k$ is an optimal solution to \eqref{expm-subw}
 if and only if $w_i^k\in\!\partial\psi^*(\rho\|x_{\!_{\mathcal{J}_i}}^k\|)$.
 Since $\psi^*$ is a convex function in $\mathbb{R}$, the subdifferential
 $\partial\psi^*(\rho\|x_{\!_{\mathcal{J}_i}}^k\|)$ is easily characterized by \cite{Roc70};
 for example, for the function $\phi_1$, it holds that
 \[
   \partial\psi_1^*(\rho\|x_{\!_{\mathcal{J}_i}}^k\|)=\left\{\min\Big(1, \max\Big(\frac{\varphi_1(1)\rho\|x_{\!_{\mathcal{J}_i}}^k\|-1}{a-1},0\Big)\Big)\right\}.
 \]
  Thus, the main computation work in each iterate of the GEP-MSCRA is to solve \eqref{expm-subx}.

  \medskip
  \noindent
  {\bf(b)} When $k\ge 2$, since $w_i^k\in\!\partial\psi^*(\rho\|x_{\!_{\mathcal{J}_i}}^k\|)$ for $i=1,\ldots,m$,
 we have $1-w_i^{k}\in\partial\theta(\rho\|x_{\!_{\mathcal{J}_i}}^k\|)$ for all $i$,
 which means that the subproblem \eqref{expm-subx} for $k\ge 2$ has a similar form to
 the one yielded by applying the linear approximation technique in \cite{Zou06} to
 $\frac{1}{2n}\|Ax-b\|^2+\lambda\Theta(x)$. Together with part (a), the GEP-MSCRA
 is analogous to the LLA algorithm in \cite{Zou06} for nonconvex penalized LS problems
 except the start-up and the weights. We see that the initial subproblem of the GEP-MSCRA
 depends explicitly on the dual variable $w^0$, while the initial subproblem
 of the LLA algorithm depends implicitly on an initial estimator $x^0$.
 This means that the start-up of the GEP-MSCRA is more easily controlled.

 \medskip
 \noindent
 {\bf(c)} By following the first part of proofs for Theorem \ref{local-theorem},
 when an iterate $x^k$ satisfies $\langle e-w^{k-1},\mathcal{G}(x^k)\rangle=0$,
 $x^k$ is a local optimal solution of \eqref{G-Sparse}, and then $(x^k,\overline{w}^k)$
 with $\overline{w}^k=\max({\rm sign}(\|\mathcal{G}(x^k)\|),t_{\phi}^*e)$
 is locally optimal to \eqref{pvc-GS} for any $\rho>0$ by Theorem \ref{local-theorem} .
 \end{remark}

%---------------------------------------------------------------------------------Section4
 \section{Statistical guarantees}\label{sec4}

  For convenience, throughout this section, we denote by $\overline{S}$
  the group support of the true $\overline{x}$, i.e.,
  $\overline{S}:=\{i: \overline{x}_{\!_{\mathcal{J}_i}}\ne 0\}$,
  and write $\overline{r}=|\overline{S}|$.
  With $\overline{S}$ and an integer $l>0$, we define
  \[
    \mathcal {C}(\overline{S},l)
    :=\bigg\{z\in \mathbb{R}^p\!: \exists S\supset\overline{S}\ {\rm with}\ |S|\leq l\ {\rm such\ that}\
    \sum_{i\in S^c}\|z_{\!_{\mathcal{J}_i}}\|\leq \frac{2}{1-\overline{t}_{\phi}}\sum_{i\in S}\|z_{\!_{\mathcal{J}_i}}\|\bigg\}.
  \]
  Recall that the matrix $A$ is said to satisfy the RSC of constant
  $\kappa>0$ in a set $\mathcal{C}$ if
  \[
    \frac{1}{2n}\|Ax\|^2\ge\kappa\|x\|^2\quad\ \forall x\in\mathcal{C}.
  \]
  In this section, under an RSC assumption on $A$ over
  $\mathcal {C}(\overline{S},1.5\overline{r})$, we shall establish
  an error bound for the iterate $x^k$ to $\overline{x}$ and
  verify that the error sequence is strictly decreasing as $k$ increases,
  and if in addition the nonzero group vectors of $\overline{x}$ are
  not too small, the iterate $x^k$ of the GEP-MSCRA after finite steps
  satisfies ${\rm supp}(x^k)={\rm supp}(\overline{x})$.
  Throughout the analysis, we assume that the components of the noise vector $\varepsilon$
 are independent (not necessarily identically distributed) sub-Gaussians, i.e.,
 the following assumption holds.
 \begin{assumption}\label{assump}
  Assume that $\varepsilon_i\,(i=1,\ldots,m)$ are independent (but not necessarily identically
  distributed) sub-Gaussians, i.e., there exists $\sigma\!>\!0$ such that for all $i$ and $t\!\in\! \mathbb{R}$
  \[
   \mathbb{E}[\exp(t\varepsilon_i)]\leq\exp(\sigma^2t^2/2).
  \]
 \end{assumption}
  The proofs of the main results of this section are all included in Appendix C.

%--------------------------------------------------------------------------------------------
 \subsection{Theoretical performance bounds}\label{subsec4.1}

  First of all, we characterize the error bound for every iterate $x^k$
  to the true vector $\overline{x}$.
%---------------------------------------------------------------------------------------------
 \begin{theorem}\label{errbound1}
  Let $\widehat{\varepsilon}:=\frac{1}{n}A^\mathbb{T}\varepsilon$.
  Suppose that $A$ has the RSC of constant $\kappa$ over
  $\mathcal {C}(\overline{S},1.5\overline{r})$, and
  \(
  \frac{\sqrt{3}(4-2\overline{t}_{\phi})(1-t_{\phi}^*)}{(3-\overline{t}_{\phi})\kappa}\leq\nu\leq \frac{1-\overline{t}_{\phi}}{(3-\overline{t}_{\phi})\|\mathcal{G}(\widehat{\varepsilon})\|_\infty}.
  \)
  If $\frac{\nu(3-\overline{t}_{\phi})\|\mathcal{G}(\widehat{\varepsilon})\|_\infty}{1-\overline{t}_{\phi}}
  \le\rho\le\sqrt{\frac{(3-\overline{t}_{\phi})\kappa\nu}{\sqrt{3}(4-2\overline{t}_{\phi})(1-t_{\phi}^*)}}$,
  then
  \begin{equation}\label{xkbound}
    \|x^k\!-\!\overline{x}\|
    \le \left\{\begin{array}{cl}
    \frac{\nu^{-1}(4-2\overline{t}_{\phi})}{\kappa(3-\overline{t}_{\phi})}\sqrt{1.5\overline{r}}
    &{\rm if}\ k=1;\\
    \frac{\rho\nu^{-1}(4-2\overline{t}_{\phi})}{\kappa(3-\overline{t}_{\phi})}\sqrt{1.5\overline{r}}
    &{\rm if}\ k=2,3,\ldots,.
    \end{array}\right.
  \end{equation}
 \end{theorem}
 \begin{remark}\label{remark-bound}
 {\bf(a)} When $k=1$, the subproblem \eqref{expm-subx} reduces to
 the $\ell_{2,1}$-regularized least squares problem, and the bound
 in \eqref{xkbound} has the same order as the one in
 \cite[Corollary 4]{Negahban12} except that the coefficient $2$ there
 is improved to be $\frac{\sqrt{1.5}(4-2\overline{t}_{\phi})}{3-\overline{t}_{\phi}}$.
 From the choice interval of $\nu$, the worst bound of $x^1$
 is $\frac{\sqrt{0.5\overline{r}}}{1-t_{\phi}^*}$ and that of $x^k$
 for $k\ge 2$ is $\frac{\rho\sqrt{0.5\overline{r}}}{1-t_{\phi}^*}$.

 \medskip
 \noindent
 {\bf(b)} The restriction on $\nu$ and $\rho$ implies that $\lambda^{k-1}\geq \frac{(3-\overline{t}_{\phi})\|\mathcal{G}(\widehat{\varepsilon})\|_\infty}{1-\overline{t}_{\phi}}$.
 Such a restriction on $\lambda^{k-1}$ is also required in the analysis of
 the $\ell_{2,1}$-regularized LS estimator \cite{Negahban12,Lounici-AOS}.
 The choice interval of $\nu$ depends on the RSC property of $A$ in
 $\mathcal {C}(\overline{S},1.5\overline{r})$ and the noise level.
 Clearly, for those problems in which $A$ has a better RSC property over
 $\mathcal {C}(\overline{S},1.5\overline{r})$ or the noise
 $\|\mathcal{G}(\widehat{\varepsilon})\|_\infty$ is smaller,
 there is a larger choice interval for the parameter $\nu$.
 In addition, those $\phi$ with larger $t_{\phi}^*$ and smaller $\overline{t}_{\phi}$
 can deliver a larger choice interval of $\nu$.

 \medskip
 \noindent
 {\bf(c)} If $\frac{(3-\overline{t}_{\phi})\nu\|\mathcal{G}(\widehat{\varepsilon})\|_\infty}{1-\overline{t}_{\phi}}
 \ge \nu L_{\!f}\frac{(1-t_{\phi}^*)\phi_{-}'(1)}{1-\overline{t}_{\phi}}$
 or equivalently
 $L_{\!f}\le \frac{(3-\overline{t}_{\phi})\|\mathcal{G}(\widehat{\varepsilon})\|_\infty}{(1-t_{\phi}^*)\phi_{-}'(1)}$,
 the choice interval of $\rho$ in Theorem \ref{errbound1} is included in
 $[\overline{\rho},+\infty)$. In this case, each subproblem \eqref{expm-subx}
 is a convex approximation of the exact penalty problem \eqref{pvc-GS} in
 a low dimensional space.
 \end{remark}

 Theorem \ref{errbound1} provides an error bound for every iterate of the GEP-MSCRA,
 but it is unclear whether the error bound for the current iterate $x^k$ is better
 than that of the previous iterate $x^{k-1}$, i.e., the error bound sequence
 is decreasing or not. We resolve this problem by bounding $(1-w^{k-1}_i)^2$
 for $i\in\overline{S}$ with $\mathbb{I}_{\Delta}(i)$
 where $\Delta:=\big\{i:\|\overline{x}_{\!_{\mathcal{J}_i}}\|\le 2\phi'_-(1)/\rho\big\}$.
 The following theorem states this main result, and its proof involves
 the index sets
     \begin{equation}\label{Fk}
   F^0:=\overline{S}\ \ {\rm and}\ \
   F^{k}:=\bigg\{i: \big|\|x^{k}_{\!_{\mathcal {J}_i}}\|-\|\overline{x}_{\!_{\mathcal {J}_i}}\|\big|
            \ge \frac{1}{(1\!-\!t_{\phi}^*)\rho}\bigg\}\ \ {\rm for}\ k=1,2,\ldots.
 \end{equation}

  \vspace{-0.5cm}
%-------------------------------------------------------------------------------
 \begin{theorem}\label{errbound2}
  Suppose that $A$ has the RSC of constant $\kappa$ over the set
  $\mathcal {C}(\overline{S},1.5\overline{r})$. If the parameter
  $\nu$ and $\rho$ are chosen in the same way as in Theorem \ref{errbound1},
  then for each $k\in\mathbb{N}$,
  \begin{equation}\label{bound-ineq}
   \|x^k-\overline{x}\|\le\frac{\sqrt{3}\|[\mathcal{G}(\widehat{\varepsilon})]_{\overline{S}}\|}{\kappa(\sqrt{3}\!-\!1)}
       +\frac{\sqrt{3}\rho}{\kappa(\sqrt{3}\!-\!1)\nu}\sqrt{{\textstyle\sum_{i\in \overline{S}}}\,\mathbb{I}_{\Delta}(i)}
       +\Big(\frac{1}{\sqrt{3}}\Big)^{k-1}\|x^{1}-\overline{x}\|.
  \end{equation}
 \end{theorem}

  The error bound in \eqref{bound-ineq} consists of three terms:
  the first term is the statistical error induced by the noise,
  the second one is the identification error related to the choice of
  $\rho$ and $\nu$, and the third one is the computation error.
  As will be shown in Subsection \ref{subsec4.2}, the identification error
  will become zero if the parameters $\rho$ and $\nu$ are appropriately chosen.
  Thus, inequality \eqref{bound-ineq} implies that as $k$ increases
  the error bound sequence is decreasing, and it will decrease to
  the statistical error $\|[\mathcal{G}(\widehat{\varepsilon})]_{\overline{S}}\|$
  if the parameters $\rho$ and $\nu$ are appropriately chosen,
  and otherwise it will decrease to the sum of the statistical error
  and the identification error. From \eqref{bound-ineq}, we also see that
  a smaller error bound of $x^1$ is beneficial to reduce the error bounds
  of $x^k$ for $k\ge 2$. In practice, since $\|\mathcal{G}(\widehat{\varepsilon})\|_\infty$
  is unknown, one may replace $\|\mathcal{G}(\widehat{\varepsilon})\|_\infty$ with
  $\|\mathcal{G}(Ax^1-b)\|_\infty$ to estimate the choice interval of $\rho$.
  This means that the error bound of $x^1$ is important to the choice of $\rho$.

  \medskip

  From \cite{Raskutti10} or \cite[Page 549]{Negahban12}, we know that for a design matrix
  $Z\in\mathbb{R}^{n\times p}$ from the $\Sigma$-Gaussian ensemble
  (i.e., $Z$ is formed by independently sampling each row $Z^i\sim N(0,\Sigma)$),
  there exists a constant $\kappa>0$ (depending on the positive definite matrix $\Sigma$)
  such that $Z$ has the RSC over $\mathcal{C}(\overline{S},l)$ with probability
  greater than $1\!-\!c_1\exp(-c_2n)$ as long as $n>c\sum_{i\in\overline{S}}|\mathcal{J}_i|\log p$,
  where $c,c_1$ and $c_2$ are absolutely positive constants.
  Together with Theorem \ref{errbound2} and Lemma \ref{noise-result1} in Appendix A,
  we immediately get the following result.
%------------------------------------------------------------------------------------------
 \begin{corollary}\label{corollary-errbound}
  Suppose Assumption \ref{assump} holds and
  \(
  \frac{\sqrt{3}(4-2\overline{t}_{\phi})(1-t_{\phi}^*)}{(3-\overline{t}_{\phi})\kappa}\leq\nu\leq \frac{1-\overline{t}_{\phi}}{(3-\overline{t}_{\phi})K},
  \)
  where $K=\frac{\sigma}{n}\sqrt{2\log\frac{2p}{\eta}\|\mathcal{J}\|_\infty}\|A\|_{2,\infty}$
  for some $\eta\in(0,1)$.
  If $\frac{\nu(3-\overline{t}_{\phi})K}{1-\overline{t}_{\phi}}
  \le\rho\le\sqrt{\frac{(3-\overline{t}_{\phi})\kappa\nu}{\sqrt{3}(4-2\overline{t}_{\phi})(1-t_{\phi}^*)}}$,
  then as long as $n>\mathcal{O}(\sum_{i\in\overline{S}}|\mathcal{J}_i|\log p)$,
  for each $k\in\mathbb{N}$ the following inequality
  \[
   \|x^{k}\!-\overline{x}\|\leq\frac{\sqrt{3}}{\kappa(\sqrt{3}\!-\!1)}\Big(K\sqrt{\overline{r}}
   +\frac{\rho}{\nu}\sqrt{{\textstyle\sum_{i\in \overline{S}}}\,\mathbb{I}_{\Delta}(i)}\Big)
   +\Big(\frac{1}{\sqrt{3}}\Big)^{k-1}\|x^{1}\!-\overline{x}\|
  \]
  holds with probability at least $1-\eta$.
 \end{corollary}

%-------------------------------------------------------------------------------
 \subsection{Group selection consistency}\label{subsec4.2}

 In this part, we shall show that in finite steps the GEP-MSCRA can deliver
 an output $x^l$ satisfying ${\rm supp}(\mathcal{G}(x^l))={\rm supp}(\mathcal{G}(\overline{x}))$
 if the nonzero group vectors of $\overline{x}$ is not too small.
 To this end, we need to assume that the following least squares solution
 belongs to $\Omega$:
 \begin{equation}\label{defxls}
  x^{{\rm LS}}\in\mathop{\arg\min}_{x\in\mathbb{R}^p}\left\{\frac{1}{2n}\|Ax-b\|^2:\
   {\rm supp}(\mathcal{G}(x))\subseteq \mathcal{J}_{\overline{S}}\right\}.
 \end{equation}
 For the solution $x^{{\rm LS}}$, one may establish the following
 $\ell_{\infty}$-norm error bound result.
  %-----------------------------------------------------------------------------------
  \begin{lemma}\label{bound-xLS}
   Suppose that $A$ has the RSC of constant $\kappa$ over the set
   $\mathcal {C}(\overline{S},1.5\overline{r})$. Then,
  \begin{equation}\label{xLS-err}
   \|\mathcal{G}(x^{{\rm LS}}-\overline{x})\|_\infty\le\|\mathcal{G}(\widehat{\varepsilon}^\dag)\|_{\infty}
   \ \ {\rm with}\ \widehat{\varepsilon}^\dagger:=(A^\mathbb{T}_{\mathcal{J}_{\overline{S}}}
   A_{\mathcal{J}_{\overline{S}}})^{-1}A^\mathbb{T}_{\mathcal{J}_{\overline{S}}}\varepsilon.
 \end{equation}
 \end{lemma}
 \begin{proof}
  When the matrix $A$ satisfies the RSC over the set
  $\mathcal{C}(\overline{S},l)$ for some $l>\overline{r}$, we have
  \begin{equation}\label{ASbar}
   \sigma_{\min}(A_{\!\mathcal{J}_{\overline{S}}})/\sqrt{n}\ge\sqrt{2\kappa}
   \end{equation}
  and hence $A_{\mathcal{J}_{\overline{S}}}$ has full column rank and $\widehat{\varepsilon}^\dagger$
  is well defined. Here $\sigma_{\min}(A_{\!\mathcal{J}_{\overline{S}}})$ is the smallest
  singular value of the matrix $A_{\!\mathcal{J}_{\overline{S}}}$.
  Indeed, for any $x\in \mathbb{R}^p$ with ${\rm supp}(\mathcal{G}(x))\subseteq \overline{S}$,
  we have
  \(
   \frac{1}{2n}\|Ax\|^2=\frac{1}{2n}\|A_{\!\mathcal{J}_{\overline{S}}}x_{\!_{\mathcal{J}_{\overline{S}}}}\|^2
   \ge\frac{1}{2n}\sigma_{\min}(A_{\!\mathcal{J}_{\overline{S}}})^2
   \|x_{\!_{\mathcal{J}_{\overline{S}}}}\|^2=\frac{1}{2n}\sigma_{\min}(A_{\!\mathcal{J}_{\overline{S}}})^2\|x\|^2,
  \)
  which along with $x\in\mathcal{C}(\overline{S},l)$ implies that $\kappa\le\frac{1}{2n}[\sigma_{\min}(A_{\!\mathcal{J}_{\overline{S}}})]^2$,
  i.e., the inequality \eqref{ASbar} holds. Now by the optimality of $x^{{\rm LS}}$ to
  the problem \eqref{defxls}, we have $A_{\mathcal {J}_{\overline{S}}}^\mathbb{T}(Ax^{{\rm LS}}-b)=0$.
  For $j\in \mathcal {J}_{\overline{S}}$,
  \begin{align}
   \big|x^{{\rm LS}}_j-\overline{x}_j\big|
   &=|e^\mathbb{T}_j(A_{\mathcal {J}_{\overline{S}}}^\mathbb{T}A_{\mathcal{J}_{\overline{S}}})^{-1}
   A_{\mathcal {J}_{\overline{S}}}^\mathbb{T}(A_{\mathcal{J}_{\overline{S}}}\overline{x}_{\!_{\mathcal {J}_{\overline{S}}}}
   -A_{\mathcal {J}_{\overline{S}}}x^{{\rm LS}}_{\!_{\mathcal {J}_{\overline{S}}}})|\nonumber\\
   &=\big|e^\mathbb{T}_j(A_{\mathcal{J}_{\overline{S}}}^\mathbb{T}A_{\mathcal{J}_{\overline{S}}})^{-1}
     A_{\mathcal{J}_{\overline{S}}}^\mathbb{T}(A\overline{x}-b+b-Ax^{{\rm LS}})\big|\\
   &=\big|e^\mathbb{T}_j(A_{\mathcal{J}_{\overline{S}}}^\mathbb{T}A_{\mathcal {J}_{\overline{S}}})^{-1}
     A_{\mathcal{J}_{\overline{S}}}^\mathbb{T}\varepsilon\big|.\nonumber
 \end{align}
  Along with $x_j^{{\rm LS}}=0$ and $\overline{x}_j=0$ for $j\notin\mathcal{J}_{\overline{S}}$,
  we immediately obtain \eqref{xLS-err}.
  \end{proof}

 Now we are ready to state the group selection consistency of the GEP-MSCRA.
%---------------------------------------------------------------------------------
 \begin{theorem}\label{theorem-consistency}
  Suppose that the matrix $A$ has the RSC of constant $\kappa$ over the set
  $\mathcal {C}(\overline{S},1.5\overline{r})$ and
  \(
  \frac{(1-t^*_\phi)(1+3\sqrt{5})}{4\kappa}\le\nu
  \le\frac{1-\overline{t}_{\phi}}{2\|\mathcal{G}(\varepsilon^{\rm{LS}})\|_\infty}
  \)
  with $\varepsilon^{\rm{LS}}\!:=\frac{1}{n}A^\mathbb{T}(Ax^{\rm{LS}}\!-b)$.
  If $\rho$ is chosen such that
  $\max\!\big(\frac{2\phi'_-(1)}{\min_{i\in\overline{S}}\|\overline{x}_{\!_{\mathcal {J}_i}}\!\|},
  \frac{2\nu\max(\|\mathcal{G}(\varepsilon^{\rm{LS}})\|_\infty,
  2\kappa\|\mathcal{G}(\widehat{\varepsilon}^\dag)\|_\infty)}{1-\overline{t}_{\phi}}\big)
  <\rho\le\!\sqrt{\frac{4\kappa\nu}{(1-t_{\phi}^*)(1+3\sqrt{5})}}$,
  then for each $k\in\mathbb{N}$
  \begin{align}\label{errk}
  &\|x^k-x^{\rm{LS}}\|\le \frac{\max(1,\sqrt{5}\rho/2)}{\nu\kappa}\sqrt{|F^{k-1}|},\nonumber\\
   &\sqrt{|F^{k}|}\le\frac{\max(1,\rho)\rho(1\!-\!t_{\phi}^*)(1\!+\!3\sqrt{5})}{6\nu\kappa}\sqrt{|F^{k-1}|}.
  \end{align}
  Also, $x^k=x^{\rm{LS}}$ and ${\rm supp}(\mathcal{G}(x^{k}))=\overline{S}$
  for $k\ge\overline{k}:=\lceil\frac{0.5\ln(\overline{r})}
  {\ln(6\nu\kappa)-\ln[(\max(1,\rho)\rho(1-t_{\phi}^*)(1+3\sqrt{5}))]}\rceil+1$.
  \end{theorem}
 \begin{remark}\label{remark-consistency}
  {\bf (a)} Theorem \ref{theorem-consistency} shows that if the parameters
  $\nu$ and $\rho$ are appropriately chosen, then the iterate $x^k$ with
  $k>\overline{k}$ coincides with the oracle solution $x^{\rm LS}$ and
  its group support coincides with $\overline{S}$. Similar to Remark \ref{remark-bound}(b),
  for those problems in which $A$ has a better RSC property in
  $\mathcal {C}(\overline{S},1.5\overline{r})$ and the noise
  $\|\mathcal{G}(\varepsilon^{\rm{LS}})\|_\infty$ is smaller,
  the choice interval of $\nu$ is larger. If the smallest nonzero
  group vector of $\overline{x}$ is suitable large, say
  \(
    \min_{i\in\overline{S}}\|\overline{x}_{\!_{\mathcal {J}_i}}\!\|
    \ge\frac{\phi_{-}'(1)(1\!-\!t_{\phi}^*)}{\nu\max(\|\mathcal{G}(\varepsilon^{\rm{LS}})\|_\infty,
  2\kappa\|\mathcal{G}(\widehat{\varepsilon}^\dag)\|_\infty)},
  \)
  then the choice of $\rho$ depends only on the noise. It is not hard to
  observe that those $\phi$ with smaller $\overline{t}_{\phi}$
  and larger $t_{\phi}^*$ lead to a larger choice interval of $\nu$
  and $\rho$ and a smaller $\overline{k}$. Together with Remark \ref{remark-bound}(b),
  the GEP-MSCRA with such $\phi$ is better in terms of the error bound
  and the group consistency.

  \medskip
  \noindent
  {\bf(b)} By Lemma \ref{noise-result2}, we have
  $\|\mathcal{G}(\varepsilon^{\rm{LS}})\|_\infty\le K$
  w.p. at least $1-\eta$ for $\eta\in(0,1)$.
  We next show that $\kappa\|\mathcal{G}(\widehat{\varepsilon}^\dag)\|_\infty\le K$
  w.p. no less than $1-\eta$ for $\eta\in(0,1)$.
  Indeed, by Lemma \ref{LS-noise},
  \[
    \kappa\|\mathcal{G}(\widehat{\varepsilon}^\dag)\|_\infty
    \le \frac{\kappa Kn}{\sigma_{\min}(A_{\mathcal {J}_{\overline{S}}})\|A\|_{2,\infty}}
    \le \frac{K\sqrt{\kappa n}}{\sqrt{2}\|A\|_{2,\infty}}.
  \]
  In addition, since for any $e_j\in \mathbb{R}^p$ with $j=1,2,\ldots,p$,
  we have $e_j\in\mathcal{C}(\overline{S},1.5\overline{r})$
  which, together with $\frac{1}{2n}\|Ae_j\|^2=\frac{1}{2n}\|A_j\|^2\|e_j\|^2$,
  implies that $\sqrt{\kappa}\le\frac{1}{\sqrt{2n}}\|A\|_{2,\infty}$.
  Substituting this relation into the last inequality yields that
  $\kappa\|\mathcal{G}(\widehat{\varepsilon}^\dag)\|_\infty\le K$.
  Thus, $\|\mathcal{G}(\varepsilon^{\rm{LS}})\|_\infty$ and
  $\kappa\|\mathcal{G}(\widehat{\varepsilon}^\dag)\|_\infty$
  have the upper bound of the same order in a high probability.
 \end{remark}

  Using Lemma \ref{noise-result2}-\ref{LS-noise}, Remark \ref{remark-consistency}(b)
  and Theorem \ref{theorem-consistency}, we obtain the following result.
%---------------------------------------------------------------------------------
 \begin{corollary}\label{corollary-consistency}
  Suppose that Assumption \ref{assump} holds and
  \(
  \frac{(1-t^*_\phi)(1+3\sqrt{5})}{4\kappa}<\nu\leq\frac{1-\overline{t}_{\phi}}{2K}.
  \)
  If
  \(
  \max\!\Big(\frac{2\phi'_-(1)}{\min_{i\in\overline{S}}\|\overline{x}_{\!_{\mathcal {J}_i}}\!\|},
  \frac{4K\nu}{1-\overline{t}_{\phi}}\Big)\!<\rho\le\!\sqrt{\frac{4\kappa\nu}{(1-t_{\phi}^*)(1+3\sqrt{5})}},
  \)
  then as long as $n>\mathcal{O}(\sum_{i\in\overline{S}}|\mathcal{J}_i|\log p)$,
  we have $x^k=x^{\rm{LS}}$ and ${\rm supp}(\mathcal{G}(x^{k}))=\overline{S}$
  for $k\ge\overline{k}$ w.p. at least $1-2\eta$ for $\eta\in(0,1)$.
  \end{corollary}

  Corollary \ref{corollary-errbound} and \ref{corollary-consistency} provide
  the theoretical guarantees in statistical sense. We need to point out,
  when a similar column normalization condition is imposed to the design
  matrix $A$, one may follow the analysis in \cite{Negahban12} to improve
  the probability bound results.

%-------------------------------------------------------------------------Section 7
  \section{Numerical experiments for the GEP-MSCRA}\label{sec5}

  The GEP-MSCRA consists in solving a sequence of weighted $\ell_{2,1}$-norm regularized
  problems. The key to its implementation is to develop an effective
  solver to \eqref{expm-subx} or equivalently
  \begin{equation}\label{Eprob-sub}
   \min_{x,u\in\mathbb{R}^p,z\in\mathbb{R}^n}\Big\{\frac{1}{2}\|z\|^2
   +{\textstyle\sum_{i=1}^{m}}\omega_i\|x_{\!_{\mathcal {J}_i}}\|+\delta_{\Omega}(u):\ Ax-z=b,\,x-u=0\Big\},
  \end{equation}
  where $\omega_i\!=n\lambda(1-\!w_i^k)$ for $i=1,\ldots,m$ are nonnegative weights.
  There are some solvers developed for the unconstrained counterpart of
  \eqref{expm-subx}; for example, the LARS-type algorithm in \cite{Yau17},
  the R-package {\bf gglasso} developed by Yang and Zou \cite{YangZou15}
  with the groupwise-majorization-descent algorithm,
  the Matlab package {\bf SLEP} developed by Liu and Ye \cite{LiuYe09} with
  the accelerated proximal gradient method \cite{Nesterov07},
  and the semismooth Newton ALM developed by Li, Sun and Toh \cite{LiSunToh16}.
  The first three solvers are solving \eqref{expm-subx} with $\Omega=\mathbb{R}^p$,
  while the last one is solving its dual problem. These solvers can not be
  applied directly to the problem \eqref{Eprob-sub}
  since it involves an additional nonsmooth term $\delta_\Omega(u)$.
%---------------------------------------------------------------------------------------- Section 5.1
 \subsection{Implementation of the GEP-MSCRA}\label{sec5.1}

  Motivated by the good performance of the semismooth Newton ALM (see \cite{LiSunToh16,SunYangToh15}),
  we shall develop it for solving the dual of \eqref{Eprob-sub} which takes the following form
  \begin{equation}\label{Edprob-sub}
   \min_{\xi\in\mathbb{R}^n,\eta,\zeta\in\mathbb{R}^p}\left\{\frac{1}{2}\|\xi\|^2+\langle b,\xi\rangle+R\|\eta\|_1+\delta_{\Lambda}(\zeta)
   :\ A^{\mathbb{T}}\xi+\eta-\zeta=0\right\},
  \end{equation}
  where $\Lambda=\Lambda_1\times\Lambda_2\times\cdots\times\Lambda_m$ with
  $\Lambda_i:=\{z\in\mathbb{R}^{|\mathcal {J}_i|}\ |\ \|z\|\le \omega_i\}$
  for $i=1,2,\ldots,m$. For a given $\sigma>0$, the augmented Lagrangian function
  of problem \eqref{Edprob-sub} is defined as
  \begin{equation*}
   L_\sigma(\eta,\xi,\zeta;x)
   \!:=\frac{1}{2}\|\xi\|^2\!+\langle b,\xi\rangle+\!R\|\eta\|_1+\delta_{\Lambda}(\zeta)
    +\langle x,A^{\mathbb{T}}\xi\!+\eta-\zeta\rangle +\frac{\sigma}{2}\|A^{\mathbb{T}}\xi\!+\eta-\zeta\|^2.
  \end{equation*}
  The iteration steps of the augmented Lagrangian method for \eqref{Edprob-sub} is described as follows.

 %---------------------------------------------------------------------------------------
 \begin{algorithm}[t]
 \caption{\label{SNAL}{\bf\ \ An inexact ALM for the dual problem \eqref{Edprob-sub}}}
 \textbf{Initialization:} Choose $\sigma_0>0$ and a starting point $(\eta^0,\xi^0,\zeta^0,x^0)$. Set $j=0$.\\
 \textbf{while} the stopping conditions are not satisfied \textbf{do}
 \begin{enumerate}
  \item  Solve the following nonsmooth convex minimization problem inexactly
         \begin{equation}\label{ALM-subprob}
                  (\eta^{j+1},\xi^{j+1},\zeta^{j+1})
                  \approx\mathop{\arg\min}_{\xi\in\mathbb{R}^{n},\eta,\zeta\in\mathbb{R}^{p}}
                   L_{\sigma_j}(\eta,\xi,\zeta;x^j).
          \end{equation}
  \item Update the multiplier by the formula
                \(
                  x^{j+1}=x^j+\sigma_{j}(A^{\mathbb{T}}\xi^{j+1}+\eta^{j+1}-\zeta^{j+1}).
                \)
  \item Update $\sigma_{j+1}\uparrow\sigma_\infty\leq\infty$. Set~$j\leftarrow j+1$, and then go to Step 1.
 \end{enumerate}
  \textbf{end while}
  \end{algorithm}	

  Observe that the augmented Lagrangian subproblem \eqref{ALM-subprob} is a two-block
  nonsmooth convex program. We use the accelerated block coordinate descent (ABCD) method
  to seek $(\eta^{j+1},\xi^{j+1},\zeta^{j+1})$ in \eqref{ALM-subprob}.
  The iterations of the ABCD method are described below.
 %---------------------------------------------------------------------------------------
 \begin{algorithm}[!h]
 \caption{\label{ABCD}{\bf\ \ An ABCD for solving the Lagrangian subproblem \eqref{ALM-subprob}}}
 \textbf{Initialization:} Choose the initial point $(\widetilde{\xi}^1,\widetilde{\zeta}^1)=(\xi^j,\zeta^j)$
        and let $t_1=1$. Set $k:=1$.\\
 \textbf{while} the stopping conditions are not satisfied \textbf{do}
 \begin{enumerate}
  \item  Compute the following minimization problems
         \begin{subnumcases}{}\label{ABCD-eta}
                \eta^{k,j}=\mathop{\arg\min}_{\eta\in\mathbb{R}^p}L_{\sigma_{\!j}}(\eta,\widetilde{\xi}^k,\widetilde{\zeta}^k;x^j),\\
                \label{ABCD-xizeta}
                (\xi^{k,j},\zeta^{k,j})=\mathop{\arg\min}_{\xi\in\mathbb{R}^{n},\zeta\in\mathbb{R}^{p}}L_{\sigma_{\!j}}(\eta^{k,j},\xi,\zeta;x^j).
         \end{subnumcases}
  \item Set~$t_{k+1}=\frac{1+\sqrt{1+4t_k^2}}{2}$ and $\beta_k=\frac{t_{k}-1}{t_{k+1}}$, and then compute
        \[
          \widetilde{\xi}^{k+1}=\xi^{k,j}+\beta_k(\xi^{k,j}-\xi^{k-1,j})\ \ {\rm and}\ \
         \widetilde{\zeta}^{k+1}=\zeta^{k,j}+\beta_k(\zeta^{k,j}-\zeta^{k-1,j}).
        \]

  \item Let $k\leftarrow k+1$, and go to Step 1.
\end{enumerate}
\textbf{end while}
\end{algorithm}

 Let ${\rm prox}_{\ell_1,\gamma}\!:\mathbb{R}^p\to\mathbb{R}^p$ denote the proximal
 mapping of $\ell_1$-norm of parameter $\gamma$, i.e.,
 \[
   {\rm prox}_{\ell_1,\gamma}(z):=\min_{x'\in\mathbb{R}^p}\Big\{\frac{1}{2}\|x'-z\|^2+\gamma\|x'\|_1\Big\}.
 \]
 From the definition of the augmented Lagrangian function, the solution $\eta^{k,j}$ has the form
  \[
    \eta^{k,j}={\rm prox}_{\ell_1,{R}/{\sigma_{\!j}}}
    \big(\widetilde{\zeta}^k-\!A^{\mathbb{T}}\widetilde{\xi}^k-{x^j}/{\sigma_{\!j}}\big).
  \]
  Let
  \(
    \Phi_{k,j}(\xi):=\min_{\zeta\in\mathbb{R}^p}L_{\sigma_{\!j}}(\eta^{k,j},\xi,\zeta;x^j)
  \)
  for $\xi\in\mathbb{R}^n$. It is not difficult to verify that
  \[
    \xi^{k,j}=\mathop{\arg\min}_{\xi\in\mathbb{R}^n}\Phi_{k,j}(\xi)\ \ {\rm and}\ \
    \zeta^{k,j}=\Pi_{\Lambda}\big(A^{\mathbb{T}}\xi^{k,j}\!+\!\eta^{k,j}\!+x^j/\sigma_{\!j}\big).
  \]
  After an elementary calculation, one may obtain the expression of $\Phi_{k,j}$ as follows
  \[
    \Phi_{k,j}(\xi)=\frac{\sigma_{\!j}}{2}
    \Big\|\Pi_{\Lambda}\Big(A^{\mathbb{T}}\xi\!+\!\eta^{k,j}\!+\!\frac{x^j}{\sigma_{\!j}}\Big)
     \!-\!\Big(A^{\mathbb{T}}\xi\!+\!\eta^{k,j}\!+\!\frac{x^j}{\sigma_{\!j}}\Big)\Big\|^2
     \!+\!\frac{1}{2}\|\xi\|^2\!+\!\langle b,\xi\rangle+R\|\eta^{k,j}\|_1.
  \]
  By the strong convexity of $\Phi_{k,j}$, $\xi^{k,j}=\mathop{\arg\min}_{\xi\in\mathbb{R}^n}\Phi_{k,j}(\xi)$
  iff $\xi^{k,j}$ satisfies the system
  \begin{equation}\label{nonsmooth-system}
    \nabla\Phi_{k,j}(\xi)=b+\xi+\sigma_{\!j} A\left[\!\Big(A^{\mathbb{T}}\xi\!+\!\eta^{k,j}\!+\!\frac{x^j}{\sigma_{\!j}}\Big)
    -\Pi_{\Lambda}\Big(A^{\mathbb{T}}\xi\!+\!\eta^{k,j}\!+\!\frac{x^j}{\sigma_{\!j}}\Big)\right]=0.
  \end{equation}
  The system \eqref{nonsmooth-system} is strongly semismooth (see \cite{Mifflin77,QiSun93,SunSun02}
  for the related discussion), and we apply the semismooth Newton method for solving it.
  Write $y\!:=A^{\mathbb{T}}\xi+\eta^{k,j}+\frac{x^j}{\sigma_{\!j}}$.
  By \cite[Proposition 2.3.3 \& Theorem 2.6.6]{Clarke83},
  the Clarke Jacobian $\partial\nabla\Phi_{k,j}$ of $\Phi_{k,j}$ satisfies
  \begin{equation}\label{inclusion}
    \partial(\nabla\Phi_{k,j})(\xi)\subseteq \widehat{\partial}^2\Phi_{k,j}(\xi)
    := I + \sigma_jA(I-\partial\Pi_{\Lambda}(y))A^{\mathbb{T}}
  \end{equation}
  where $\widehat{\partial}^2\Phi_{k,j}$ is the generalized Hessian of $\Phi_{k,j}$ at $\xi$.
  Since the exact characterization of $\partial\nabla\Phi_{k,j}$ is difficult to obtain,
  we replace $\partial\nabla\Phi_{k,j}$ with $\widehat{\partial}^2\Phi_{k,j}$ in
  the solution of \eqref{nonsmooth-system}. Let

  \medskip
  \noindent
  $W\in\partial\Pi_{\Lambda}(y)$. By \cite[Theorem 2.6.6]{Clarke83}, we know that
  $W={\rm Diag}(W_{\!_{\mathcal{J}_1}},\ldots, W_{\!_{\mathcal{J}_m}})$
  is a block diagonal matrix with the $i$th block
  $W_{\!_{\mathcal{J}_i}}\!\in\partial\Pi_{\Lambda_i}(y_{\!_{\mathcal{J}_i}})$,
  where $\partial\Pi_{\Lambda_i}(y_{\!_{\mathcal{J}_i}})$ takes the form of
  \begin{equation}
  \partial\Pi_{\Lambda_i}(y_{\!_{\mathcal{J}_i}})
  =\begin{cases}\label{CJPI}
    \qquad\qquad \{I\} &\text{if}~\|y_{\!_{\mathcal{J}_i}}\|<\omega_i,\\
   {\rm conv}\big(I,I-\frac{1}{\omega_i^2}y_{\!_{\mathcal{J}_i}}y_{\!_{\mathcal{J}_i}}^{\mathbb{T}}\big)
   &\text{if}~\|y_{\!_{\mathcal{J}_i}}\|=\omega_i,\\
   \Big\{\omega_i\Big(\frac{1}{\|y_{\!_{\mathcal{J}_i}}\|}I-\frac{1}{\|y_{\!_{\mathcal{J}_i}}\|^3}
   y_{\!_{\mathcal{J}_i}}y_{\!_{\mathcal{J}_i}}^{\mathbb{T}}\Big)\Big\}
   &\text{if}~\|y_{\!_{\mathcal{J}_i}}\|>\omega_i.
  \end{cases}
 \end{equation}
  Here, ${\rm conv}\big(I,I-\frac{1}{\omega_i^2}y_{\!_{\mathcal{J}_i}}y_{\!_{\mathcal{J}_i}}^{\mathbb{T}}\big)$
  means the convex combination of $I$ and $I-\frac{1}{\omega_i^2}y_{\!_{\mathcal{J}_i}}y_{\!_{\mathcal{J}_i}}^{\mathbb{T}}$.
  From \eqref{inclusion} and \eqref{CJPI}, each element
  $I+\sigma_jA(I\!-\!W)A^{\mathbb{T}}$ in $\widehat{\partial}^2\Phi_{k,j}(\xi)$
  is positive definite, which by \cite{QiSun93} implies that
  the following semismooth Newton method has a fast convergence rate.
%-------------------------------------------------------------------------------------
 \begin{algorithm}[!h]
 \caption{\label{SNCG}{\bf\ \ A semismooth Newton-CG (SNCG) algorithm for \eqref{nonsmooth-system}}}
 \textbf{Initialization:} Choose $\overline{\theta}\!\in(0,1),\tau\in(0,1),\delta\in(0,1),\mu\in\!(0,\frac{1}{2})$
               and $\xi^0\in\!\mathbb{R}^n$. Set $l=0$.\\
 \textbf{while} the stopping conditions are not satisfied \textbf{do}
 \begin{enumerate}
  \item  Choose a matrix $V^l\in\widehat{\partial}^2\Phi_{k,j}(\xi^l)$.
               Solve the following linear system
               \begin{equation*}\label{SNCG-dj}
                 V^ld=-\nabla\Phi_{k,j}(\xi^l)
              \end{equation*}
              with the conjugate gradient (CG) algorithm to find $d^l$ such that
              \[
                \|V^ld^l+\nabla\Phi_{k,j}(\xi^l)\|\le\min(\overline{\theta},\|\nabla\Phi_{k,j}(\xi^l)\|^{1+\tau})
              \]
  \item Set $\alpha_l=\delta^{m_l}$, where $m_l$ is the first nonnegative integer $m$ for which
               \begin{equation*}
                \Phi_{k,j}(\xi^l+\delta^md^l)\leq\Phi_{k,j}(\xi^l)+\mu\delta^m\langle\nabla\Phi_{k,j}(\xi^l),d^l\rangle.
               \end{equation*}

  \item Set $\xi^{l+1}=\xi^l+\alpha_ld^l$ and $l\leftarrow l+1$, and then go to Step 1.
\end{enumerate}
\textbf{end while}
\end{algorithm}	

  \medskip

  During the implementation of the semismooth Newton ALM for \eqref{Edprob-sub},
  we terminated the algorithm once
  \(
    \max\{\varepsilon_{{\rm pinf}}^j,\varepsilon_{{\rm dinf}}^j,\varepsilon_{{\rm gap}}^j\}\le \epsilon^j,
  \)
  where $\varepsilon_{{\rm gap}}^j$ is the primal-dual gap, i.e., the sum of the objective values
  of \eqref{Eprob-sub} and \eqref{Edprob-sub} at $(\eta^{j},\xi^{j},\zeta^{j})$,
  and $\varepsilon_{{\rm pinf}}^j$ and $\varepsilon_{{\rm dinf}}^j$ are the primal and dual infeasibility
  measure at $(\eta^{j},\xi^{j},\zeta^{j})$, respectively, defined as follows
  \[
    \varepsilon_{{\rm pinf}}^j
    :=\frac{\sigma_{\!j-1}\|(\zeta^{k,j}-\widetilde{\zeta}^k)+A^{\mathbb{T}}(\widetilde{\xi}^{k}-\xi^{k,j})\|}{1+\|b\|}
    \ \ {\rm and}\ \
    \varepsilon_{{\rm dinf}}^{j}:=\frac{\|x^{j}-x^{j-1}\|}{\sigma_{\!j-1}}.
  \]

  Now we return to the choice of parameters in the GEP-MSCRA.
  Taking into account the choice of $\rho$ in the first stage may not be the best,
  we use a dynamic adjustment for $\rho$ during the test. Specifically,
  we choose $\rho^1=\!\frac{2}{\|\mathcal{G}(x^{1})\|_{\infty}}$ and increase it
  by the rule $\rho^k = \min(2\rho^{k-1},10^8/\|\mathcal{G}(x^k)\|_\infty)$ for $k\ge 2$.
  The choice of $\nu$ is specified in the experiments. By Remark \ref{remark-alg}(c),
  we terminate the GEP-MSCRA at the iterate $x^k$ whenever it satisfies
  \[
    \big\langle e-w^{k-1},\mathcal{G}(x^k)\big\rangle\le \epsilon_{\rm gap}\ \ {\rm or}\ \
   \frac{|f(x^k)-f(x^{k-1})|}{\max(1,f(x^k))}\le\epsilon_{\rm loss},\,
   |\|x^k\|_{a,0}-\|x^{k-1}\|_{a,0}|\le 1
  \]
  where
  \(
    \|z\|_{a,0}:=\sum_{i=1}^m\!\mathbb{I}_{\{i:\,\|z_{\!_{\mathcal{J}_i}}\|>10^{-6}\}}(z)
  \)
  means the approximate group zero-norm of $z$.
  During the testing, we choose $\epsilon_{\rm gap}=10^{-6}$ and $\epsilon_{\rm loss}=10^{-2}$,
  and solve the subproblem \eqref{expm-subx} by Algorithm \ref{SNAL}
  with the tolerance $\epsilon^j=\max(10^{-5},0.8\epsilon^{j-1})$
  and $\epsilon^{0}=0.1\epsilon_{\rm loss}$.
  All numerical results of this section are obtained from a laptop running
  on 64-bit Windows Operating System with an Intel(R) Core(TM) i7-7700 CPU 2.8GHz
  and 16 GB memory.

%--------------------------------------------------------------------------------- Section 5.1
 \subsection{Numerical experiments for group sparse regressions}\label{sec5.2}

  We shall evaluate the performance of the GEP-MSCRA in the group sparse regression
  setting by using the simulated data. We generate the simulation data with the sample size $n$,
  the dimension of variables $p$, the number of groups $m$,
  and the dimension of each group $d=\lceil p/m\rceil$. The matrix $A$ is generated randomly
  by one of the following ways:
   \begin{itemize}
   \item[(I)] $A=\textbf{randn}(n,p)$;

   \item[(II)] $ A = \textbf{sign}(\textbf{rand}([n,p]) - 0.5)$;  $\textrm{ind} =\textrm{find}(A =0)$;
               $ A(\textrm{ind}) = \textrm{ones}(\textrm{size}(\textrm{ind}))$;

  \item[(III)] $A\!=\textbf{hadamard}(n)$; $\textrm{picks}=\textrm{randperm}(n)$;
                $\textrm{picks}=\!\textrm{sort}(\textrm{picks}(1\!:\!n))$;
                $A\!=\!A(\textrm{picks},\!:)$.
  \end{itemize}
  We select $\overline{r}$ groups randomly from $m$ groups, say $\{m_1,\ldots,m_{\overline{r}}\}$,
  as the support of $\overline{x}$, and generate the entries of
  $\overline{x}_{\!_{\mathcal{J}_{i}}}$ for $i\in\{m_1,\ldots,m_{\overline{r}}\}$
  in one of the following seven ways:
  \begin{itemize}
   \item[(i)] $\overline{x}_{\!_{\mathcal{J}_{i}}}=\alpha\,\textbf{randn}(|\mathcal{J}_{i}|,1)$
              for $i\in\{m_1,\ldots,m_{\overline{r}}\}$\ \ {\rm with}\ $\alpha=2$ or $10^5$;

   \item[(ii)] $\overline{x}_{\!_{\mathcal{J}_{i}}}=\alpha\,\textbf{rand}(|\mathcal{J}_{i}|,1)-0.5$
                for $i\in\{m_1,\ldots,m_{\overline{r}}\}$\ \ {\rm with}\ $\alpha=2$ or $10^5$;

  \item[(iii)] $\overline{x}_{\!_{\mathcal{J}_{i}}}=\alpha\,\textbf{sign}(\textbf{randn}(|\mathcal{J}_{i}|,1))$
               for $i\in\{m_1,\ldots,m_{\overline{r}}\}$\ \ {\rm with}\ $\alpha=1$ or $10^5$;

  \item[(iv)] $\overline{x}_{\!_{\mathcal{J}_{i}}}=-\frac{10^5}{\sqrt{i}}e$ for
              $i\in\{m_1,\ldots,m_{\overline{r}/2}\}$ and $\overline{x}_{\!_{\mathcal{J}_{i}}}=\frac{10^5}{\sqrt{i}}e$
              for $i\in\{m_{(\overline{r}+1)/2},\ldots,m_{\overline{r}}\}$.
  \end{itemize}
  Then, we set $b=A(\overline{x}+\vartheta_1\frac{\widetilde{\varepsilon}}{\|\widetilde{\varepsilon}\|})
  +\vartheta_2\frac{\varepsilon}{\|\varepsilon\|}$
  where $\widetilde{\varepsilon}=\textbf{randn}(p,1)$, $\varepsilon=\textbf{randn}(p,1)$
  and $\vartheta_1$ and $\vartheta_2$ are the nonnegative constants representing the scale
  of the noise vectors $\varepsilon$ and $\widetilde{\varepsilon}$.
  Since the true $\overline{x}$ is known for these synthetic problems,
  we take $R=1000\|\overline{x}\|_\infty$ for the set $\Omega$.
  We find from experiments that Algorithm \ref{SNAL}
  is not sensitive to the value of $R$.

%-------------------------------------------------------------------------------
  \subsubsection{Performance of the GEP-MSCRA with different $\phi$}\label{subsub5.1.1}

  This part aims to evaluate the performance of the GEP-MSCRA with
  $\phi\in\{\phi_1,\phi_2,\phi_3,\phi_4\}$ where $a=3.7$ and $a=3$
  are used for $\phi_1$ and $\phi_2$ respectively, and $\epsilon=10^{-2}$
  is used for $\phi_4$. With the design matrix $A\in\mathbb{R}^{n\times p}$
  of type I for $(p,m,\overline{r})=(2^{12},256,10)$,
  we generate $10$ test problems randomly as above for every type of $\overline{x}$
  with $(\vartheta_1,\vartheta_2)=(0.1,0.1)$, and apply the GEP-MSCRA for
  solving the test problems with $\lambda=({0.1}/{n})\|A^{\mathbb{T}}b\|_\infty$.
  Figure \ref{fig1} plots the average relative
  prediction error curve and the average computing time curve, respectively,
  yielded by the GEP-MSCRA with each $\phi$ under the sample size
  $n=\lfloor\frac{p}{\beta}\rfloor$ for $\beta\in\{5,6,\ldots,17\}$.
  Here, for each sample size, the average relative error and computing time
  is the average of the total relative prediction error and computing time
  of the $70$ test problems. The relative error is defined by
  \(
    \textbf{relerr}:=\frac{\|x^{\rm out}-\overline{x}\|}{\|\overline{x}\|}
  \)
  where $x^{\rm out}$ is the output.
%-------------------------- Please insert Figure 5.1-------------------------------------
% \begin{figure}[htbp]
%  \setlength{\abovecaptionskip}{1pt}
% \begin{center}
% {\includegraphics[width=1.1\textwidth]{figure_1.eps}}\\  %[width=15cm,height=6.0cm]
% \caption{\small Performance of the GEP-MSCRA with $\phi_1$-$\phi_4$ under different sample size}
% \label{fig1}
% \end{center}
% \end{figure}

\begin{figure}[ht]
\setlength{\abovecaptionskip}{1pt}
\begin{center}
\includegraphics[width=14cm,height=5.0cm]{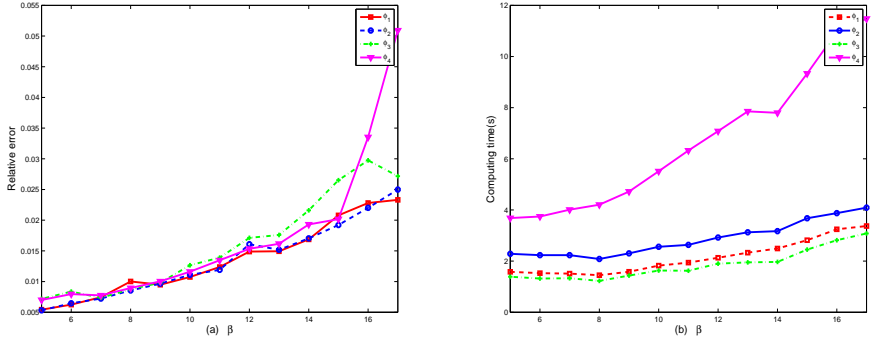}
\end{center}
\caption{\small Performance of the GEP-MSCRA with $\phi_1$-$\phi_4$ under different sample size}\label{fig1}
\end{figure}

 \medskip

 Figure \ref{fig1} shows that the relative errors yielded with $\phi_1$-$\phi_4$
 are comparable, but those yielded with $\phi_3$ and $\phi_4$ have a little
 bigger fluctuation. In addition, the GEP-MSCRA with $\phi_2$ and $\phi_4$
 requires more computing time than the GEP-MSCRA with $\phi_1$ and $\phi_3$ does.
 By this, we choose the GEP-MSCRA with $\phi_1$ for the subsequent experiments.

  \subsubsection{Numerical comparison with the SLEP}\label{subsub5.1.1}

  The SLEP is a solver to the subproblem \eqref{expm-subx}
  without the constraint $x\in\Omega$ but with positive
  weights. So, we first compare the performance of
  Algorithm \ref{SNAL} for solving the subproblem \eqref{expm-subx}
  for $k=1$ and $w^0=0$ with that of the SLEP for solving its counterpart without
  the constraint $x\in\Omega$ under different $\lambda$. Unless otherwise stated,
  all parameters involved in the SELP are set to be the default one.
  We generate $10$ test problems randomly as above for every type of $\overline{x}$
  with $(\vartheta_1,\vartheta_2)=(0.1,0.1)$ and the design matrix
  $A\in\mathbb{R}^{n\times p}$ of type I for $(p,m,\kappa)=(2^{12},256,15)$
  and $n=\lfloor{p}/{10}\rfloor$. Figure \ref{fig2} plots the average relative error
  and computing time curves of Algorithm \ref{SNAL} and the SLEP for solving
  the $70$ problems with $\lambda=(\beta/n)\|A^{\mathbb{T}}b\|_\infty$.
  We see that the relative error yielded by Algorithm \ref{SNAL}
  has less variation than the one yielded by the SLEP when
  $\lambda\in[0.03/n,0.3/n]\|A^{\mathbb{T}}b\|_\infty$, which means that
  it is easier to choose an appropriate $\lambda$ for Algorithm \ref{SNAL}.
  Since the problem \eqref{expm-subx} is more difficult than its
  unconstrained counterpart, Algorithm \ref{SNAL} requires more
  time than the SLEP does, but its computing time decreases as
  $\lambda$ increases, and when $\lambda\ge (0.2/n)\|A^{\mathbb{T}}b\|_\infty$
  its time is less than three times that of the SLEP.
 %-------------------------- Please insert Figure 5.1-------------------------------------
% \begin{figure}[!h]
%  \setlength{\abovecaptionskip}{1pt}
% \begin{center}
% {\includegraphics[width=1.1\textwidth]{figure_2.eps}}\\  %[width=15cm,height=6.0cm]
% \caption{\small Performance of Algorithm \ref{SNAL} and the SLEP under different $\lambda=(\beta/n)\|A^{\mathbb{T}}b\|_\infty$}
% \label{fig2}
% \end{center}
% \end{figure}

\begin{figure}[ht]
\setlength{\abovecaptionskip}{1pt}
\begin{center}
\includegraphics[width=14cm,height=5.0cm]{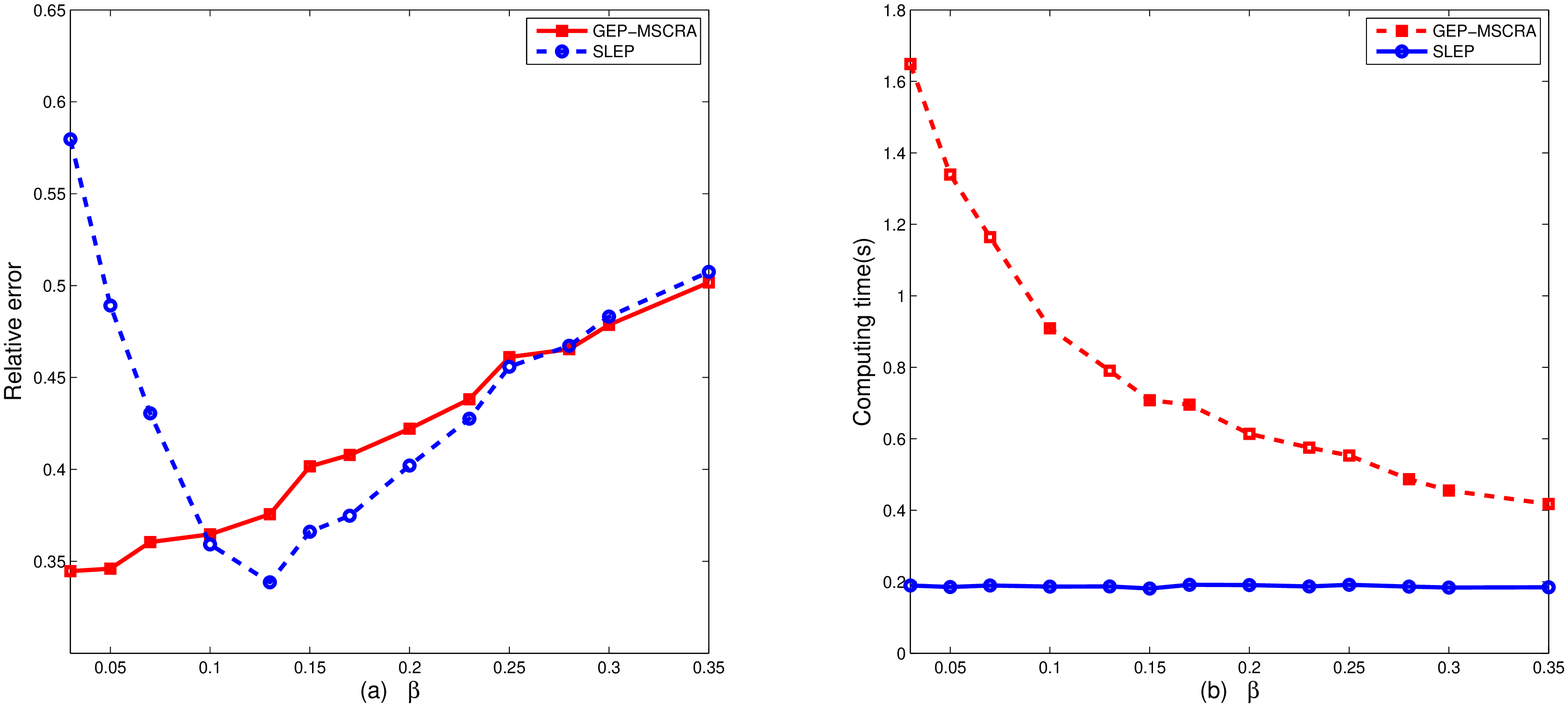}
\end{center}
\caption{\small Performance of Algorithm \ref{SNAL} and the SLEP under different $\lambda=(\beta/n)\|A^{\mathbb{T}}b\|_\infty$}\label{fig2}
\end{figure}

  \medskip

  Next we compare the performance of the GEP-MSCRA for computing $\widehat{x}$
  with that of the SLEP for computing the $\ell_{2,1}$-norm regularized LS estimator,
  i.e., the one defined by the subproblem \eqref{expm-subx} with $k=1$ and $w^0=0$
  but without the constraint $x\in\Omega$. To this end, for each type of $A$ with
  $(p,m,\kappa)=(2^{12},256,15)$, we generate $10$ test problems randomly
  for every type of $\overline{x}$ with $(\vartheta_1,\vartheta_2)=(0.1,0.3)$,
  and then apply the GEP-MSCRA and the SLEP, respectively, for solving
  the corresponding test problems. By Figure \ref{fig2},
  we choose $\lambda=(0.1/n)\|A^{\mathbb{T}}b\|_\infty$ for the GEP-MSCRA
  and $\lambda=(0.13/n)\|A^{\mathbb{T}}b\|_\infty$ for the SLEP.
  Figure \ref{fig3} plots the average relative error and computing time
  curves under different sample size $n=\lfloor\frac{p}{\beta}\rfloor$
  for $\beta\in\{3,4,\ldots,15\}$. From Figure \ref{fig3}, we see that
  although the SLEP is faster than the GEP-MSCRA, for the matrix $A$ of
  type I and II, the relative error of its output is about \textbf{six}
  or \textbf{seven} times higher than that of the GEP-MSCRA,
  and for the matrix $A$ of type III, the relative error of its output is
  about \textbf{one and half} times higher than that of the GEP-MSCRA.
  In addition, Figure \ref{fig4} shows under each sample size, the group
  sparsity of the output yielded by the SLEP is much higher than that of
  $\overline{x}$ when the sample size becomes less, but that of the output
  yielded by the GEP-MSCRA is close to that of the true $\overline{x}$.
  This means that the estimator yielded by the GEP-MSCRA is much better
  than the one yielded by the SLEP in terms of the relative error and
  the group sparsity. Notice that the matrix $A$ of type I and II
  satisfies the RSC condition in a high probability.
  Thus, the numerical performance matches the theoretical analysis well.
 %-------------------------- Please insert Figure 5.1-------------------------------------
% \begin{figure}[!h]
%  \setlength{\abovecaptionskip}{1pt}
% \begin{center}
% {\includegraphics[width=16cm,height=8.0cm]{figure_3.eps}}\\  %[width=15cm,height=6.0cm]
% \caption{\small Relative error of the output yielded by the GEP-MSCRA and the SLEP}
% \label{fig3}
% \end{center}
% \end{figure}

 \begin{figure}[ht]
\setlength{\abovecaptionskip}{1pt}
\begin{center}
\includegraphics[width=14cm,height=8.5cm]{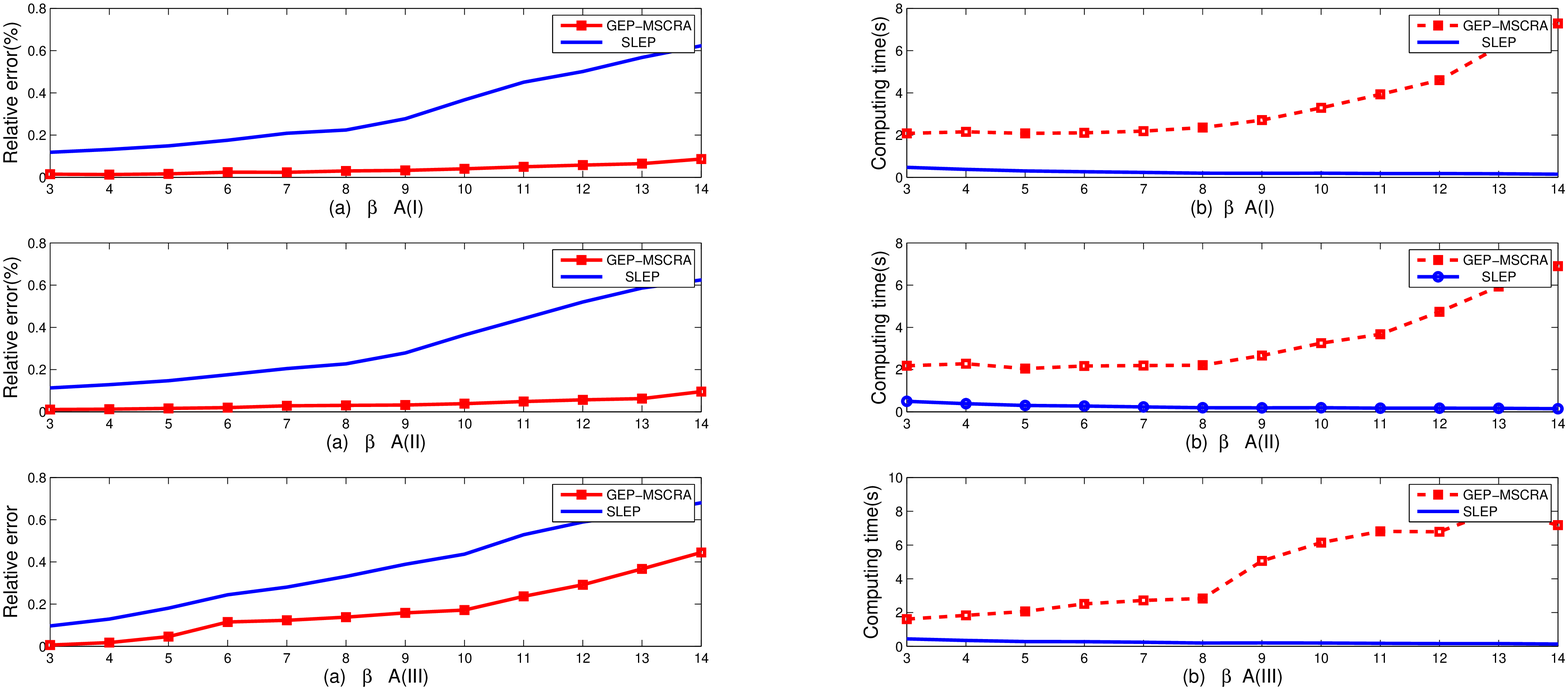}
\end{center}
\caption{\small Relative error of the output yielded by the GEP-MSCRA and the SLEP}\label{fig3}
\end{figure}

 %-------------------------- Please insert Figure 5.1-------------------------------------
% \begin{figure}[!h]
%  \setlength{\abovecaptionskip}{1pt}
% \begin{center}
% {\includegraphics[width=1.1\textwidth]{figure_4.eps}}\\  %[width=15cm,height=6.0cm]
% \caption{\small Group sparsity of the output yielded by the GEP-MSCRA and the SLEP}
% \label{fig4}
% \end{center}
% \end{figure}

\begin{figure}[ht]
\setlength{\abovecaptionskip}{1pt}
\begin{center}
\includegraphics[width=14cm,height=5cm]{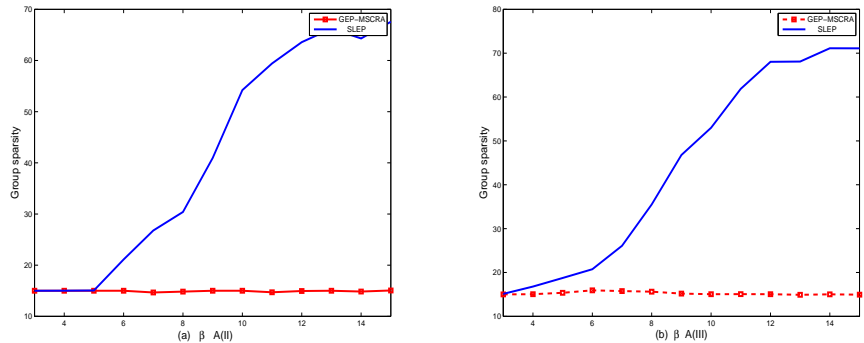}
\end{center}
\caption{\small Group sparsity of the output yielded by the GEP-MSCRA and the SLEP}\label{fig4}
\end{figure}
%--------------------------------------------------------------------------------- Section 5.1
 \subsection{Numerical experiments for multi-task learning}\label{sec5.2}

  In multi-task learning (see \cite{Argyriou-ML,Obozinski10,ZhangGY08}),
  we are given a training set of $m$ tasks $\{(a_i^k,y_i^k)\}_{i=1}^{n_k}$
  from the linear models $h_k(a)=\langle w_{\!_{\mathcal{J}_k}},a\rangle$
  for $k=1,\ldots,m$, where $w_{\!_{\mathcal{J}_k}}\in\mathbb{R}^{|\mathcal{J}_k|}$
  is the weight vector for the $k$th task, $a_i^k\in\mathbb{R}^{|\mathcal{J}_k|}$
  is the $i$th training sample for the $k$th task, $y_i^k$ is the corresponding output,
  and $n_k$ is the number of training samples for the $k$th task.
  Write $y_k=[y_1^k,\ldots,y_{n_k}^k]^{\mathbb{T}}\in\mathbb{R}^{n_k}$
  and $b=[y_1^{\mathbb{T}},\ldots,y_m^{\mathbb{T}}]^{\mathbb{T}}\in\mathbb{R}^n$
  with $n=\sum_{j=1}^mn_j$. Let $A_{\!_{\mathcal{J}_k}}=[a_1^k,\ldots,a_{n_k}^k]^{\mathbb{T}}\in\mathbb{R}^{n_k\times|\mathcal{J}_k|}$
  denote the data matrix for the $k$th task. Clearly, the model \eqref{model}
  is also applicable to the multi-task learning by replacing $\overline{x}$ with
  $w=[w_{\!_{\mathcal{J}_1}}^{\mathbb{T}},\ldots,w_{\!_{\mathcal{J}_m}}^{\mathbb{T}}]^{\mathbb{T}}$.

  \medskip

 This part focuses on the comparison of the GEP-MSCRA and
 the MALSAR\footnote{\url{http://yelab.net/software/MALSRA (version 1.1)/}}
 for a real data set (School data) from the Inner London Education
 Authority\footnote{Available at \url{http://www.mlwin.com/intro/datasets.html}}.
 Among others, the MALSAR is a solver for the unconstrained $\ell_{2,1}$-regularized
 LS model, and since the true $\overline{x}$ is unknown for this real problem,
 we take $R=2000$ for the set $\Omega$. This data set has been used
 in previous works on multi-task learning (see \cite{Evgeniou-ML}).
 It consists of examination scores of 15362 students from 139 secondary schools
 in London during the years 1985, 1986 and 1987. There are 139 tasks, corresponding to
 predicting student performance in each school. The input consists of the year of
 the examination (YR), 4 school-specific and 3 student-specific attributes,
 and each sample contains 28 attributes.

 \medskip

 We first test the prediction performance of the GEP-MSCRA and the MALSAR
 with different $\lambda=\nu^{-1}$. We generate the training and
 test sets by 10 random splits of the data, so that ${\bf 75\%}$ of the examples from
 each school (task) belong to the training set and ${\bf 25\%}$ to the test set.
 The subfigures in the first line of Figure \ref{fig5} plot the prediction error
 and time curves of two solvers with $\lambda=(0.001\beta/n)\|A^{\mathbb{T}}b\|_\infty$,
 where the solvers use the solution associated to the current $\lambda$ as
 the initial point for solving the problem associated to the next $\lambda$,
 and the subfigures in the second line are plotted by the solutions
 yielded by the GEP-MSCRA with the initial $x^0=0$ and
 the MALSAR with the default one.
 %-------------------------- Please insert Figure 5.1-------------------------------------
% \begin{figure}[!h]
% \setlength{\abovecaptionskip}{1pt}
% \begin{center}
%  {\includegraphics[width=16cm,height=8cm]{Figure_5.eps}}\\
%  \caption{\small Prediction errors yielded by the GEP-MSCRA and the SLEP under different $\lambda$}
%   \label{fig5}
% \end{center}
% \end{figure}

\begin{figure}[ht]
\setlength{\abovecaptionskip}{1pt}
\begin{center}
\includegraphics[width=14cm,height=7cm]{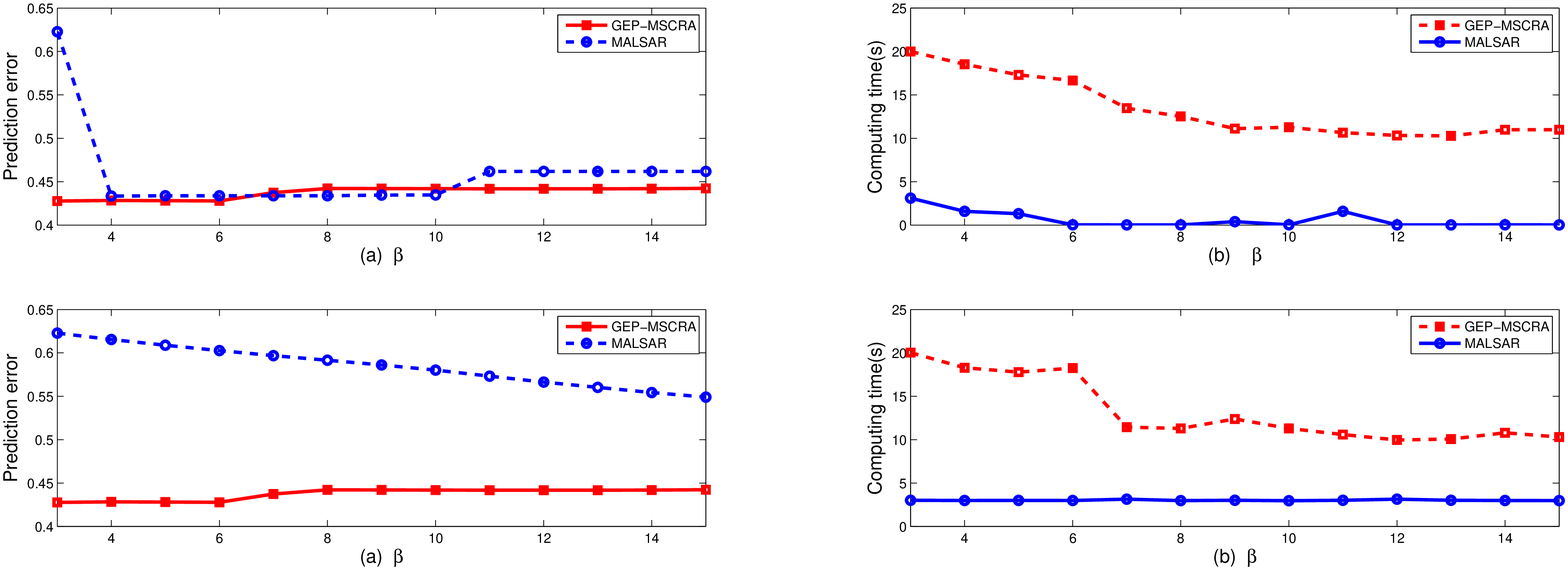}
\end{center}
\caption{\small Prediction errors yielded by the GEP-MSCRA and the SLEP under different $\lambda$}\label{fig5}
\end{figure}

 \medskip

 Figure \ref{fig5} shows that the performance of the GEP-MASCRA does not
 depend on the initial point, but that of the MALSAR
 improves much if the solution corresponding to the current $\lambda$ is used as
 the starting point for solving the problem associated to the next $\lambda$.
 The prediction error of the GEP-MASCRA is at least ${\bf 20\%}$ lower than
 that of the MALSAR when the latter does not use the solution corresponding to
 the current $\lambda$ as the starting point, and is comparable even superior to
 that of the MALSAR even if it uses the solution associated to the current
 $\lambda$ as the starting point. Also, from the left subfigure
 in Figure \ref{fig6}, the GEP-MSCRA yields the solution with better group
 sparsity than the MALSAR does; and from the right subfigure,
 the MALSAR does not yield a group sparse solution without using the solution
 associated to the current $\lambda$ as the next starting point,
 but the GEP-MSCRA yields the solution with desirable group sparsity.

 %-------------------------- Please insert Figure 5.1-------------------------------------
% \begin{figure}[!h]
% \setlength{\abovecaptionskip}{1pt}
% \begin{center}
%  {\includegraphics[width=1.1\textwidth]{Figure_6.eps}}\\
%  \caption{\small Group sparsity yielded by the GEP-MSCRA and the SLEP under different $\lambda$}
%   \label{fig6}
% \end{center}
% \end{figure}
%

 \begin{figure}[ht]
\setlength{\abovecaptionskip}{1pt}
\begin{center}
\includegraphics[width=14cm,height=5.0cm]{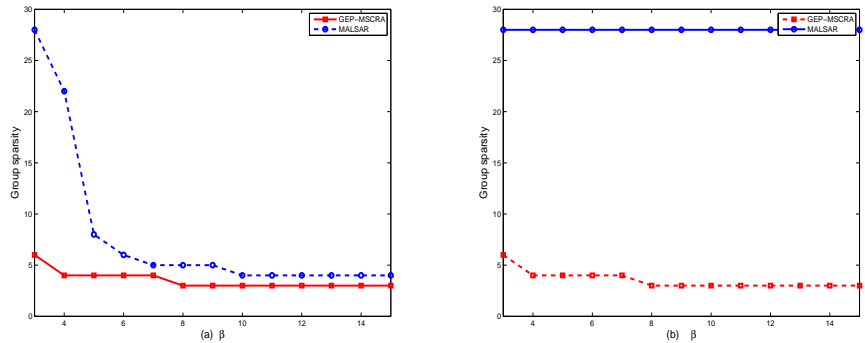}
\end{center}
\caption{\small Group sparsity yielded by the GEP-MSCRA and the SLEP under different $\lambda$}\label{fig6}
\end{figure}

 \medskip

 Next we test the prediction performance of the GEP-MSCRA and the MALSAR
 with different numbers of training samples. We generate the training and
 test sets by 10 random splits of the data so that $100\beta\%$
 of the examples from each school (task) belong to the training set
 and $100(1\!-\!\beta)\%$ to the test set. The subfigures in the first line
 of Figure \ref{fig7} plots the prediction error curves and the computing time curves
 with $\lambda=(0.005/n)\|A^{\mathbb{T}}b\|_\infty$, and the subfigures
 in the second line are plotted with $\lambda=(0.013/n)\|A^{\mathbb{T}}b\|_\infty$.
 We see that the prediction error of the GEP-MSCRA is decreasing as the number of
 training samples increases, but that of the MALSAR does not improve even
 increases as the number of training samples increases. Moreover,
 the prediction error of the GEP-MSCRA is at least lower than ${\bf 20\%}$
 that of the MALSAR when $50\%$ of the examples are used as the training set,
 and the prediction error of the former is lower than ${\bf 5\%}$ that of
 the latter when only $35\%$ of the examples are used as the training set.
 From Figure \ref{fig8}, we see that under each kind of training samples,
 the GEP-MSCRA yields the group sparsity less than ${\bf 5}$,
 but the MALSAR does not yield group sparsity under the two $\lambda$.

 %-------------------------- Please insert Figure 5.1-------------------------------------
 %\begin{figure}[!h]
% \setlength{\abovecaptionskip}{1pt}
% \begin{center}
%  {\includegraphics[width=16cm,height=8cm]{Figure_7.eps}}\\
%  \caption{\small Prediction errors of the GEP-MSCRA and the SLEP under different train samples}
%   \label{fig7}
% \end{center}
% \end{figure}
 \begin{figure}[ht]
\setlength{\abovecaptionskip}{1pt}
\begin{center}
\includegraphics[width=14cm,height=7cm]{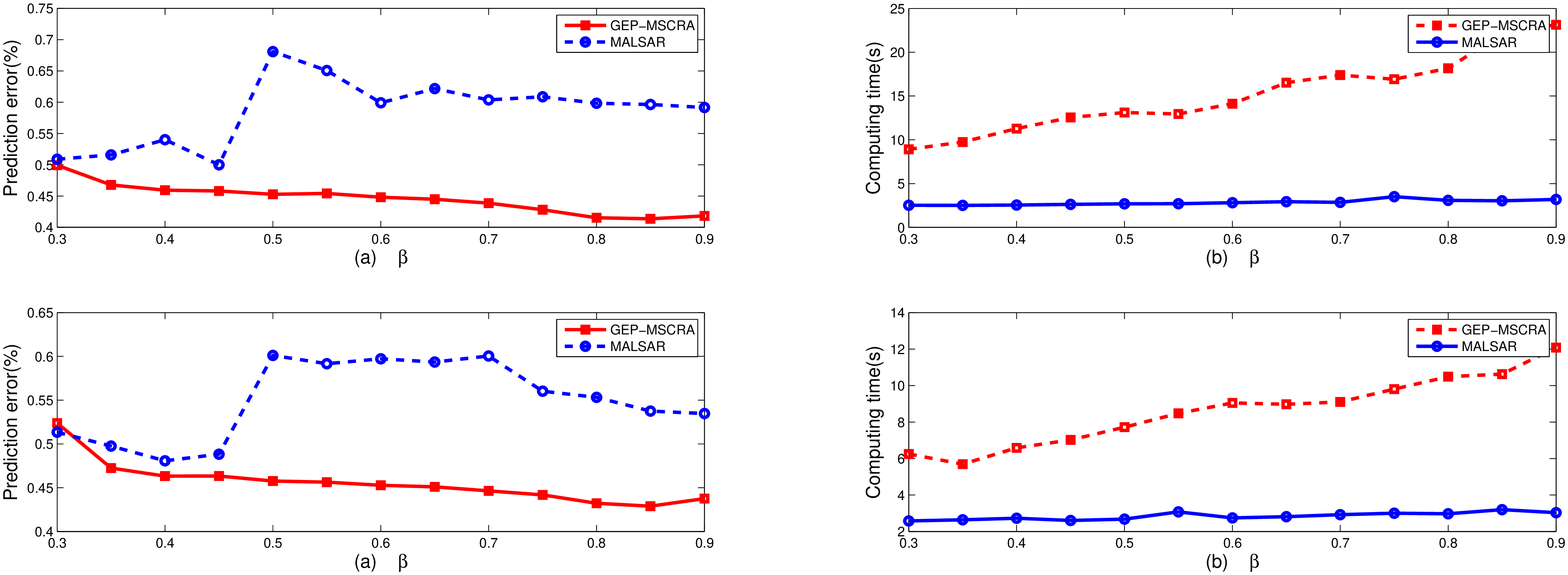}
\end{center}
\caption{\small Prediction errors of the GEP-MSCRA and the SLEP under different train samples}\label{fig7}
\end{figure}

%\begin{figure}[!h]
% \setlength{\abovecaptionskip}{1pt}
% \begin{center}
%  {\includegraphics[width=1.1\textwidth]{Figure_8.eps}}\\
%  \caption{\small Group sparsity of the GEP-MSCRA and the SLEP under different train samples}
%   \label{fig8}
% \end{center}
% \end{figure}
 \begin{figure}[ht]
\setlength{\abovecaptionskip}{1pt}
\begin{center}
\includegraphics[width=14cm,height=5.0cm]{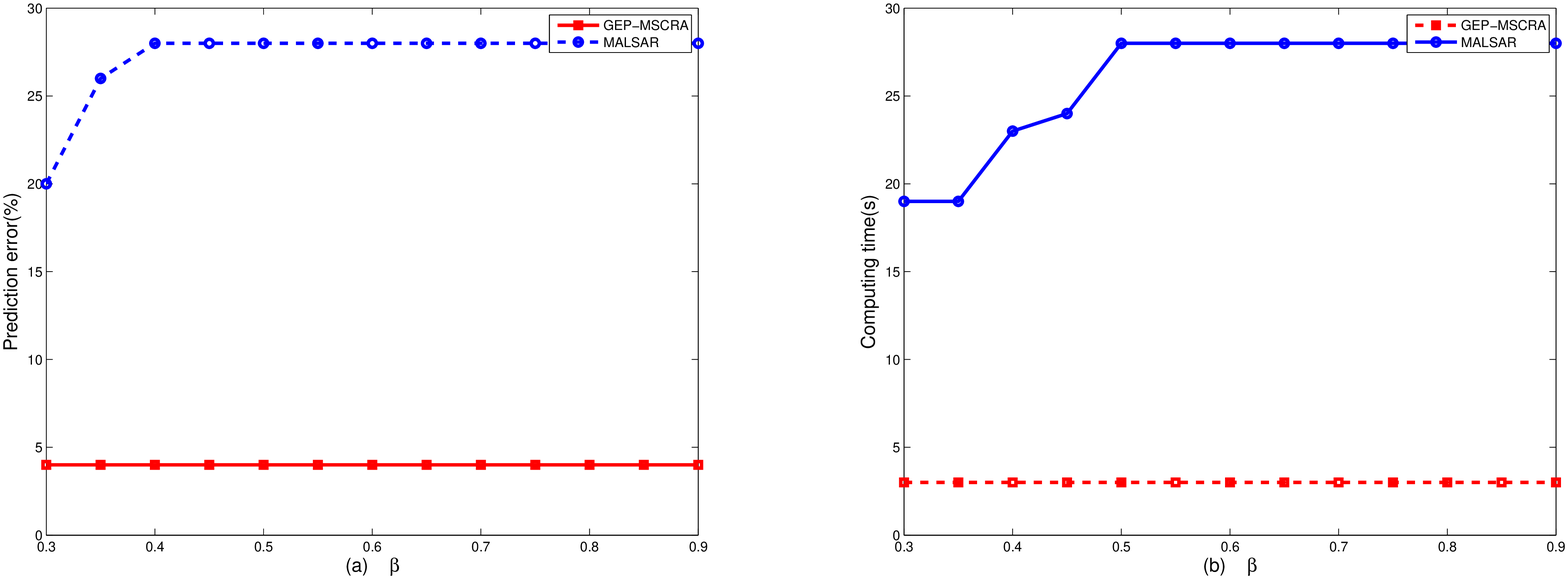}
\end{center}
\caption{\small Group sparsity of the GEP-MSCRA and the SLEP under different train samples}\label{fig8}
\end{figure}

%---------------------------------------------------------------------------------------
 \section{Conclusions}\label{conclusion}

 In this paper we showed that the group zero-norm regularized least
 squares estimator can be obtained from an exact penalization problem
 by using the equivalent MPEC of \eqref{G-Sparse} and developing
 the global exact penalty for the MPEC, and found that the popular
 SCAD and MCP penalized estimators also arise from the global exact
 penalty framework. Based on the structure of the exact penalty problem,
 we proposed a primal-dual convex relaxation approach for computing
 this estimator. For the proposed GEP-MSCRA, we provided its statistical
 guarantees and confirmed its efficiency by making comparison with
 the SLEP and the MALSAR on synthetic group sparse regression problems
 and real multi-task learning problems. In our future work, we shall
 further study the global exact penalty results for the MPEC from statistical angle,
 and develop global exact penalty results for other statistical problems
 with a certain combinatorial property.

  \bigskip
 \noindent
 {\bf\large Appendix A}

 \medskip

  The following several lemmas provide some upper estimations for
  the noise vector $\varepsilon$. Among others, Lemma \ref{noise-result1}
  follows directly by using the same arguments as \cite[Lemma 5]{TZhang10},
  and Lemma \ref{noise-result2} and \ref{LS-noise} follow from
  the same arguments as those for \cite[Lemma 3]{TZhang11}.
 %-----------------------------------------------------------------------------------------------
 \begin{alemma}\label{noise-result1}
  Let $\|\mathcal{J}\|_\infty:=\max_{1\le i\le m}\{|\mathcal{J}_i|\}$.
  Then, under Assumption \ref{assump}, for any given $\eta\in(0,1)$
  the following inequality holds with probability at least $1-\eta$:
  \[
  \|\mathcal{G}(\widehat{\varepsilon})\|_\infty
   \le \frac{\sigma}{n}\sqrt{2\|\mathcal{J}\|_\infty\log({2p}/{\eta})}\|A\|_{2,\infty}.
  \]
 \end{alemma}
%-----------------------------------------------------------------------------------------------
 \begin{alemma}\label{noise-result2}
  Suppose that $A_{\mathcal {J}_{\overline{S}}}$ has full column rank.
  Then, under Assumption \ref{assump}, for any given $\eta\in(0,1)$
  the following inequality holds with probability (w.p.) at least $1-\eta$:
  \[
  \|\mathcal{G}(\widehat{\varepsilon}^\dag)\|_\infty
  \le\frac{\sigma}{\sigma_{\min}(A_{\mathcal {J}_{\overline{S}}})}
    \sqrt{2\|\mathcal{J}\|_\infty\log({2p}/{\eta})}.
  \]
 \end{alemma}
%------------------------------------------------------------------------------------------
 \begin{alemma}\label{LS-noise}
  Define $\varepsilon^{\rm{LS}}\!:=\frac{1}{n}A^\mathbb{T}(Ax^{\rm{LS}}\!-b)$
  where $x^{\rm{LS}}$ is the solution defined by \eqref{defxls}.
  Under Assumption \ref{assump}, for any given $\eta\in(0,1)$
  the following inequality holds w.p. at least $1-\eta$:
  \[
   \varepsilon^{\rm LS}_{_{\mathcal {J}_i}}=0\ \ {\rm for}\ i\in \overline{S}\ {\rm and}\
   \|\mathcal{G}(\varepsilon^{\rm LS})\|_\infty\le \frac{\sigma\sqrt{2\|\mathcal{J}\|_\infty\log(2p/\eta)}}{n}\|A\|_{2,\infty}
   \ {\rm for}\ i\notin\overline{S}.
  \]
 \end{alemma}

 \bigskip
 \noindent
 {\bf\large Appendix B}

 \medskip

 Next we shall provide the proof of Theorem \ref{exact-penalty}.
 This requires three technical lemmas.  The first two characterize
 some important properties of the function family $\Phi$.
%---------------------------------------------------------------------------------
 \begin{alemma}\label{lemma1-phi}
  Let $\phi\in\Phi$. Then, the set $(\partial\phi)^{-1}(\frac{1}{1-t_{\phi}^*})\cap[t^*_{\phi},1)$
  is nonempty and compact.
 \end{alemma}
 \begin{proof}
  Since $[0,1]\subseteq{\rm int}({\rm dom}\phi)$, from \cite[Theorem 23.4]{Roc70}
  $\partial\phi(t)=[\phi_{-}'(t),\phi_{+}'(t)]$ is nonempty and bounded for each $t\in[0,1]$.
  We first argue that $(\partial\phi)^{-1}(\frac{1}{1-t_{\phi}^*})\cap[t^*_{\phi},1)\ne\emptyset$.
  Assume that there exists $\overline{t}\in(t^*_{\phi},1)$ such that
  $\phi_{-}'(\overline{t})<\phi_{-}'(1)$ (if not, we will have
  $\partial\phi(t)=\{\phi'(t)\}=\{\phi_{-}'(1)\}$ for all $t\in(t^*_{\phi},1)$,
  and hence there exists $\xi\in(t^*_{\phi},1)$ such that $\phi'(\xi)=\frac{\phi(1)}{1-t^*_{\phi}}$,
  which implies the desired statement). Together with the convexity of $\phi$ and
  \cite[Theorem 24.1]{Roc70}, we have $\phi_{-}'(t)\le\phi_{-}'(\overline{t})$
  for all $t\in[t^*_{\phi},\overline{t}]$. By \cite[Corollary 24.2.1]{Roc70},
  \begin{align*}
   \phi(1)&=\phi(1)-\phi(t^*_{\phi})
   =\int_{t^*_{\phi}}^{1}\phi_{-}'(t)dt
   =\int_{t^*_{\phi}}^{\overline{t}}\phi_{-}'(t)dt+\int_{\overline{t}}^{1}\phi_{-}'(t)dt\\
   &<\phi_{-}'(1)(\overline{t}-t^*_{\phi})+\int_{\overline{t}}^{1}\phi_{-}'(t)dt
   \le \phi_{-}'(1)(1-t^*_{\phi}).
  \end{align*}
  In addition, by the convexity of $\phi$,
  $\phi(1)\ge\phi(t^*_{\phi})+\phi_{+}'(t^*_{\phi})(1-t^*_{\phi})=\phi_{+}'(t^*_{\phi})(1-t^*_{\phi})$.
  Thus, $a:=\frac{\phi(1)}{1-t^*_{\phi}}=\frac{1}{1-t^*_{\phi}}\in[\phi_{+}'(t^*_{\phi}),\phi_{-}'(1))$.
  If $a=\phi_{+}'(t^*_{\phi})$, clearly, $t_{\phi}^*\in(\partial\phi)^{-1}(\frac{1}{1-t_{\phi}^*})\cap[t^*_{\phi},1)$.
  So, it suffices to consider the case $a\in(\phi_{+}'(t^*_{\phi}),\phi_{-}'(1))$.
  Now $(\partial\phi)^{-1}(a)\cap[0,1)\ne\emptyset$ (if not,
  $a\in\partial\phi(t')=[\phi_{-}'(t'),\phi_{+}'(t')]$ for $t'\ge 1$ or
  $t'<t^*_{\phi}$, which contradicts $a\in(\phi_{+}'(t^*_{\phi}),\phi_{-}'(1))$).

  \medskip

  Next we show that $(\partial\phi)^{-1}(\frac{1}{1-t_{\phi}^*})\cap[t^*_{\phi},1)$
  is compact. Fix an arbitrary $b\in(\partial\phi)^{-1}(\frac{1}{1-t_{\phi}^*})$.
  Since $(\partial\phi)^{-1}(\frac{1}{1-t_{\phi}^*})$ is compact, we only need
  to argue that $b<1$. This clearly holds by noting that
  $a=\frac{1}{1-t_{\phi}^*}\in\partial\phi(b)=[\phi_{-}'(b),\phi_{+}'(b)]$
  and $a\in[\phi_{+}'(t^*_{\phi}),\phi_{-}'(1))$.
 \end{proof}
%---------------------------------------------------------------------------------
 \begin{alemma}\label{lemma2-phi}
  Let $\phi\in\Phi$. For any given $\omega\ge 0$, define
  $\upsilon^*\!:={\displaystyle\min_{t\in[0,1]}\{\phi(t)+\omega(1-\!t)\}}$. Then,
  \[
    \left\{\begin{array}{ll}
    \upsilon^*=1&{\rm if}\ \omega\in(\phi_{-}'(1),+\infty);\\
    \upsilon^*\ge\frac{\omega(1-\overline{t}_{\phi})}{\phi_{-}'(1)(1-t^*_{\phi})}&{\rm if}\  \omega\in\big[\frac{1}{1-t^*_{\phi}},\phi_{-}'(1)\big];\\
    \upsilon^*\ge\omega(1\!-\!\overline{t}_{\phi}) &{\rm if}\  \omega\in\big[0,\frac{1}{1-t^*_{\phi}}\big).
    \end{array}\right.
  \]
 \end{alemma}
 \begin{proof}
  When $\omega>\phi_{-}'(1)$, clearly, $\upsilon^*=\phi(1)$ since
  $\phi(t)+\omega(1-t)$ is nonincreasing in $[0,1]$.
  When $\omega\in\big[0,\frac{1}{1-t^*_{\phi}}\big)$,
  since $\phi_{-}'(t)\ge\phi_{+}'(\overline{t}_{\phi})>\omega$
  for any $t>\overline{t}_{\phi}$ by Lemma \ref{lemma1-phi},
  the optimal solution $\widehat{t}$ of $\min_{t\in[0,1]}\{\phi(t)+\omega(1\!-\!t)\}$
  satisfies $\widehat{t}\le\overline{t}_{\phi}$. By the convexity of $\phi$,
  \[
   \phi(t)+\omega(1\!-t)
   \ge \phi(\widehat{t})+\omega(1\!-\widehat{t})
   \ge \omega(1-\overline{t}_{\phi})\quad\ \forall t\in[0,1].
  \]
  This shows that $\upsilon^*\ge\omega(1-\overline{t}_{\phi})$ for this case.
  When $\omega\in\big[\frac{1}{1-t^*_{\phi}},\phi_{-}'(1)\big]$,
  by Lemma \ref{lemma1-phi}
  \[
    \min_{t\in[0,1]}\Big\{\phi(t)+\frac{1}{1-t^*}(1-t)\Big\}
    =\phi(\overline{t}_{\phi})+\frac{1}{1-t^*_{\phi}}(1-\overline{t}_{\phi})
    \ge\frac{1-\overline{t}_{\phi}}{1-t^*_{\phi}}
    \ge \frac{\omega(1-\overline{t}_{\phi})}{\phi_{-}'(1)(1-t^*_{\phi})},
  \]
  where the last inequality is due to $\omega\le\phi_{-}'(1)$.
  The proof is completed.
 % it follows that
%  \begin{equation}\label{temp-ineq}
%   \phi(t)+\omega(1-t)\ge \phi(t)+\frac{1}{1-t^*_{\phi}}(1-t)\quad\ \forall t\in[0,1].
%  \end{equation}
%  If $t_0>0$, from the fact that $t_0<1$ and $\frac{1}{1-t^*_{\phi}}\in\partial\phi(t_0)$,
%  it immediately follows that
%  \[
%    \min_{t\in[0,1]}\Big\{\phi(t)+\frac{1}{1-t^*}(1-t)\Big\}
%    =\phi(t_0)+\frac{1}{1-t^*_{\phi}}(1-t_0)\ge\frac{1-t_0}{1-t^*_{\phi}}.
%  \]
%  Together with \eqref{temp-ineq} and $\varrho\omega\le\phi_{-}'(1)$,
%  we have $\upsilon^*\ge\frac{1-t_0}{1-t^*_{\phi}}
%  \ge\varrho\omega\frac{1-t_0}{\phi_{-}'(1)(1-t^*_{\phi})}$.
%  If $t_0=0$, from $\frac{1}{1-t^*_{\phi}}\in\partial\phi(t_0)$ we have
%  \(
%    \phi_{+}'(0)\ge\frac{\phi(1)}{1-t^*_{\phi}}\ge\phi(1)\ge\phi(0)+\phi_{+}'(0)\ge\phi_{+}'(0),
%  \)
%  where the third inequality is due to the convexity of $\phi$ at $[0,1]$.
%  Then, for any $t\in[0,1]$,
%  \[
%   \phi(t)+\frac{\phi(1)}{1-t^*_{\phi}}(1-t)
%  \ge\phi(0)+\phi_{+}'(0)t+\frac{\phi(1)}{1-t^*_{\phi}}(1-t)\ge\frac{1}{1-t^*_{\phi}},
%  \]
%  where the second inequality is using the convexity of $\phi$ at $[0,1]$.
%  Together with \eqref{temp-ineq}, it follows that
%  $\upsilon^*\ge\frac{1}{1-t^*_{\phi}}\ge\frac{\varrho\omega}{\phi_{-}'(1)(1-t^*_{\phi})}
%  \ge\frac{\varrho\omega(1-t_0)}{\phi_{-}'(1)(1-t^*_{\phi})}$.
%  The proof is completed.
 \end{proof}

 For every $x\in\Omega$, with a parameter $\rho>0$ we define a truncated
 vector $x^{\rho}\in\Omega$ by
 \[
   (x^{\rho})_{\!_{\mathcal {J}_i}}
   =\left\{\begin{array}{cl}
    x_{\!_{\mathcal {J}_i}} &{\rm if}\ \|x_{\!_{\mathcal {J}_i}}\|>\frac{\phi'_-(1)}{\rho},\\
               0 & {\rm otherwise}.
               \end{array}\right.
 \]
 Then, the following result holds for the objective value of \eqref{G-Sparse}
 at $x^{\rho}$ and that of \eqref{pvc-GS} at $x$.

 \vspace{-0.2cm}
%---------------------------------------------------------------------------------------------------------------Lemma 3.2
 \begin{alemma}\label{bound-penalty}
  Let $\phi\in\Phi$. Then, for any $x\in\Omega$ and $w\in[0,e]$,
  when $\rho>\overline{\rho}=\nu L_{\!f}\frac{(1-t^*_{\phi})\phi_-'(1)}{1-\overline{t}_{\phi}}$,
  \begin{equation}\label{ineq-lemma32}
     \nu f(x^{\rho})+\|\mathcal{G}(x^{\rho})\|_{0}
     \leq \nu f(x)+{\textstyle\sum_{i=1}^m}\big[\phi(w_i) +\rho(1-w_i)\|x_{\!_{\mathcal {J}_i}}\|\big],
  \end{equation}
  and moreover, $x^{\rho}=x$ and $\|\mathcal{G}(x)\|_{1}-\langle w,\mathcal{G}(x)\rangle=0$
  provided that \eqref{ineq-lemma32} becomes an equality.
 \end{alemma}
 \begin{proof}
  Fix arbitrary $x\in\Omega,\,w\in[0,e]$ and $\rho>\overline{\rho}$.
  Applying Lemma \ref{lemma2-phi} with $\omega=\rho\|x_{\!_{\mathcal {J}_i}}\|$
  for every $i\in\{1,\ldots,m\}$ delivers
  \begin{align}\label{temp-xrho}
    \left\{\begin{array}{ll}
    \phi(w_i) +\rho(1-w_i)\|x_{\!_{\mathcal {J}_i}}\|\ge 1&{\rm if}\ i\in I_1;\\
    \phi(w_i) +\rho(1-w_i)\|x_{\!_{\mathcal {J}_i}}\|\ge \frac{(1-\overline{t}_{\phi})\rho\|x_{\!_{\mathcal {J}_i}}\|}{\phi_{-}'(1)(1-t^*_{\phi})}
    &{\rm if}\  i\in I_2;\\
    \phi(w_i) +\rho(1-w_i)\|x_{\!_{\mathcal {J}_i}}\|\geq\rho\|x_{\!_{\mathcal {J}_i}}\|(1\!-\!\overline{t}_{\phi})
    &{\rm if}\  i\in I_3.
    \end{array}\right.
  \end{align}
  where $I_1:=\!\big\{i: \rho\|x_{\!_{\mathcal {J}_i}}\|>\phi_{-}'(1)\big\},\,
  I_2:=\big\{i:\rho\|x_{\!_{\mathcal{J}_i}}\|\in[\frac{1}{1-t^*_{\phi}},\phi_{-}'(1)]\big\}$
  and $I_3:=(I_1\cup I_2)^{c}$. From the expression of $x^{\rho}$ and
  $\rho>\nu L_{\!f}\frac{(1-t^*_{\phi})\phi_-'(1)}{1-\overline{t}_{\phi}}
  \geq\frac{\nu L_{\!f}}{1-\overline{t}_{\phi}}$, it immediately follows that
  \[
    \sum_{i\in I_1}\big[\phi(w_i) +\rho(1-w_i)\|x_{\!_{\mathcal {J}_i}}\|\big]
    =\|\mathcal{G}(x^\rho)\|_{0}\ \ {\rm and}\
    \sum_{i\in I_2\cup I_3 }\!\big[\phi(w_i) +\rho(1-w_i)\|x_{\!_{\mathcal {J}_i}}\|\big]
    \ge\nu L_f\|x_{\!_{\mathcal {J}_i}}\|.
  \]
  Together with $|f(x)\!-\!f(x^{\rho})|\leq L_{\!f}\|x\!-\!x^{\rho}\|$
  by the Lipschitz continuity of $f$, we have
  \begin{align}\label{temp-eq0}
   &\sum_{i=1}^m\big[\phi(w_i) +\rho(1-w_i)\|x_{\!_{\mathcal {J}_i}}\|\big]-\|\mathcal{G}(x^\rho)\|_{0}\nonumber\\
   &= \sum_{i\in I_1\cup I_2\cup I_3}\big[\phi(w_i) +\rho(1-w_i)\|x_{\!_{\mathcal {J}_i}}\|\big]-\|\mathcal{G}(x^\rho)\|_{0}\nonumber\\
   &\ge \sum_{i\in I_2\cup I_3}\nu L_{\!f}\|x_{\!_{\mathcal {J}_i}}\|=\nu L_{\!f}\|\mathcal{G}(x)-\mathcal{G}(x^{\rho})\|_1\nonumber\\
   &\ge \nu L_{\!f}\|x-x^{\rho}\|\ge \nu|f(x)-f(x^{\rho})|.
  \end{align}
  By the arbitrariness of $x\in\Omega,w\in[0,e]$ and $\rho>\overline{\rho}$,
  the first part of conclusions follows.

  \medskip

  Next we prove the second part. Since inequality \eqref{ineq-lemma32} becomes an equality, i.e.,
  \begin{equation}\label{temp-eq00}
    \nu|f(x^{\rho})-f(x)|=
    {\textstyle\sum_{i=1}^m}\big[\phi(w_i) +\rho(1-w_i)\|x_{\!_{\mathcal {J}_i}}\|\big]-\|\mathcal{G}(x^{\rho})\|_{0},
  \end{equation}
  together with inequality \eqref{temp-eq0} it immediately follows that
  \begin{align}\label{temp-eq1}
  \sum_{i\in I_2\cup I_3}\big[\phi(w_i) +\rho(1-w_i)\|x_{\!_{\mathcal {J}_i}}\|\big]
    = \sum_{i\in I_2\cup I_3}\nu L_{\!f}\|x_{\!_{\mathcal {J}_i}}\|.
  \end{align}
  Suppose on the contradiction that $x\neq x^{\rho}$.
  Then there exists an index $k\in I_2\cup I_3$ such that
  $\|x_{\!_{\mathcal {J}_k}}\|\ne 0$. By \eqref{temp-xrho} and
  $\rho>\nu L_{\!f}\frac{(1-t^*_{\phi})\phi_-'(1)}{1-\overline{t}_{\phi}}$,
  $\phi(w_k) +\rho(1-w_k)\|x_{\!_{\mathcal {J}_k}}\|>\nu L_f\|x_{\!_{\mathcal {J}_k}}\|$.
  Together with $\phi(w_i) +\rho(1-w_i)\|x_{\!_{\mathcal {J}_i}}\|\ge\nu L_{\!f}\|x_{\!_{\mathcal {J}_i}}\|$
  for all $i\in I_2\cup I_3$, we obtain
  \[
    \sum_{i\in I_2\cup I_3}\big[\phi(w_i) +\rho(1-w_i)\|x_{\!_{\mathcal {J}_i}}\|\big]
    > \sum_{i\in I_2\cup I_3}\nu L_{\!f}\|x_{\!_{\mathcal {J}_i}}\|,
  \]
  which contradicts \eqref{temp-eq1}. Substituting $x=\!x^{\rho}$
  into \eqref{temp-eq00} and using the definition of $x^{\rho}$ yields
  \[
   \sum_{i\in I_1\cup I_3}\big[\phi(w_i) +\rho(1-w_i)\|x_{\!_{\mathcal {J}_i}}\|\big]
   =\|\mathcal{G}(x^\rho)\|_{0}.
  \]
  Notice that
  \(
    \phi(w_i)\!\ge\!\phi(1) +\phi_{-}'(1)(w_i-1)\geq \phi(1)-\rho\|x_{\!_{\mathcal {J}_i}}\|(1-w_i)
  \)
  for every $i\in I_1$, and hence
  \(
   \sum_{i\in I_1}\big[\phi(w_i) +\rho(1-w_i)\|x_{\!_{\mathcal {J}_i}}\|\big]
   \ge |I_1|=\|\mathcal{G}(x)\|_0.
  \)
  Together with $\phi(w_i) +\rho(1-w_i)\|x_{\!_{\mathcal {J}_i}}\|\ge 0$ for $i\in I_3$,
  the last equality implies that $\phi(w_i)=1$ for $i\in I_1$
  and $\sum_{i\in I_3}\phi(w_i)=0$. Clearly, the latter is equivalent to saying that
  $w_i=t^*_{\phi}$ for $i\in I_3$. Now from \eqref{temp-eq0} we get
  \[
   {\textstyle\sum_{i=1}^m}\big[\phi(w_i) +\rho(1-w_i)\|x_{\!_{\mathcal {J}_i}}\|\big]
  =|I_1|+{\textstyle\sum_{i=1}^m}\rho(1-w_i)\|x_{\!_{\mathcal {J}_i}}\|
  =\|\mathcal{G}(x^\rho)\|_{0}.
  \]
  This means that $\sum_{i=1}^m\rho(1-w_i)\|x_{\!_{\mathcal {J}_i}}\|=0$.
  Thus, we complete the proof.
  \end{proof}

  \medskip
  \noindent
  {\bf\large The proof of Theorem \ref{exact-penalty}:}
  \begin{aproof}
  Fix an arbitrary $\rho>\overline{\rho}$. Let $\mathcal{S}$ be the feasible set
  of \eqref{vc-GS}, and let $\mathcal{S}_{\rho}$ be that of \eqref{pvc-GS} associated to $\rho$.
  We first prove that $\mathcal{S}^*\subseteq\mathcal{S}_{\rho}^*$.
  Fix an arbitrary $(\overline{x},\overline{w})\in\mathcal{S}^*$.
  Then, $\overline{x}$ is globally optimal to \eqref{G-Sparse} and $\|\mathcal{G}(\overline{x})\|_{0}=\sum_{i=1}^m\phi(\overline{w}_i)$.
  Let $(x,w)$ be an arbitrary point from $\mathcal{S}_{\rho}$. Assume that
   $x^\rho$ be defined as in Lemma \ref{bound-penalty}. Then
  \begin{align}
   \nu f(x)+{\textstyle\sum_{i=1}^m}\big[\phi(w_i) +\rho(1-w_i)\|x_{\!_{\mathcal {J}_i}}\|\big]
   &\ge \nu f(x^{\rho})+\|\mathcal{G}(x^{\rho})\|_{0}
   \ge \nu f(\overline{x})+\|\mathcal{G}(\overline{x})\|_{0}\nonumber\\
   &=\nu f(\overline{x})+\sum_{i=1}^m\big[\phi(\overline{w}_i) +\rho(1-\overline{w}_i)\|x_{\!_{\mathcal {J}_i}}\|\big],\nonumber
  \end{align}
   where the second inequality is due to $x^\rho\in\Omega$.
   Notice that $(\overline{x},\overline{w})\in \mathcal{S}_{\rho}$ and
   $(x,w)$ is an arbitrary point from $\mathcal{S}_{\rho}$.
   The last inequality shows that $(\overline{x},\overline{w})\in \mathcal{S}_{\rho}^*$,
   and then $\mathcal{S}^*\subseteq\mathcal{S}_{\rho}^*$.

   \medskip

  We next prove $\mathcal{S}_{\rho}^*\subseteq\mathcal{S}^*$.
  Fix an arbitrary $(\overline{x},\overline{w})\in\mathcal{S}^*_{\rho}$.
  Define $\overline{w}^{\rho}\in \mathbb{R}^m$ by
  \[
    \overline{w}_i^{\rho}:=\left\{
        \begin{array}{cl}
          1 & {\rm if}\ \rho\|\overline{x}_{\!_{\mathcal {J}_i}}\|\!>\!\phi_{-}'(1);\\
          t^*_{\phi} & {\rm if}\ \rho\|\overline{x}_{\!_{\mathcal {J}_i}}\|\!\le\!\phi_{-}'(1),
        \end{array}
      \right.\ \ {\rm for}\ \ i=1,2,\ldots,m.
 \]
  Then, we have $\sum_{i=1}^m\phi(\overline{w}_i^{\rho})=\|\mathcal{G}(\overline{x})\|_{0}$
  and $\|\mathcal{G}(\overline{x}^{\rho})\|_{1}
  =\sum_{i=1}^m\overline{w}_i^{\rho}\|\overline{x}^{\rho}_{\!_{\mathcal {J}_i}}\|$,
  where $\overline{x}^{\rho}$ is defined as in Lemma \ref{bound-penalty} with $x=\overline{x}$.
  From the results of Lemma \ref{bound-penalty}, it follows that
 \begin{align}
  &\nu f(\overline{x})+{\textstyle\sum_{i=1}^m}\big[\phi(\overline{w}_i) +\rho(1-\overline{w}_i)\|\overline{x}_{\!_{\mathcal {J}_i}}\|\big]\geq\nu f(\overline{x}^{\rho})+\|\mathcal{G}(\overline{x}^{\rho})\|_{0}\nonumber\\
  &= \nu f(\overline{x}^{\rho})+{\textstyle\sum_{i=1}^m}\big[\phi(\overline{w}_i^{\rho}) +\rho(1-\overline{w}_i^{\rho})\|\overline{x}^{\rho}_{\!_{\mathcal {J}_i}}\|\big]\nonumber\\
  &\geq \nu f(\overline{x})+{\textstyle\sum_{i=1}^m}\big[\phi(\overline{w}_i) +\rho(1-\overline{w}_i)\|\overline{x}_{\!_{\mathcal {J}_i}}\|\big],\nonumber
 \end{align}
  where the last inequality is due to $(\overline{x}^{\rho},\overline{w}_i^{\rho})\in \mathcal{S}_{\rho}$.
  The last inequality implies that
  \[
   \nu f(\overline{x})+{\textstyle\sum_{i=1}^m}\big[\phi(\overline{w}_i) +\rho(1-\overline{w}_i)\|\overline{x}_{\!_{\mathcal {J}_i}}\|\big]
   =\nu f(\overline{x}^{\rho})+\|\mathcal{G}(\overline{x}^{\rho})\|_{0}.
  \]
 Using Lemma \ref{bound-penalty} again, we have $\overline{x}=\overline{x}^\rho$ and
 $\|\mathcal{G}(\overline{x})\|_{1}-\sum_{i=1}^m\overline{w}_i\|\overline{x}_{\!_{\mathcal {J}_i}}\|=0$,
 which implies $(\overline{x},\overline{w})\in \mathcal{S}$. Now let $(x,w)$ be an arbitrary point
 from $\mathcal{S}$. Then $(x,w)\in\mathcal{S}_{\rho}$, and we have
 \begin{align}
  \nu f(x)+{\textstyle\sum_{i=1}^m}\phi(w_i)
  &=\nu f(x)+{\textstyle\sum_{i=1}^m}\big[\phi(w_i) +\rho(1-w_i)\|x_{\!_{\mathcal {J}_i}}\|\big]\nonumber\\
  &\geq \nu f(\overline{x})+{\textstyle\sum_{i=1}^m}\big[\phi(\overline{w}_i)
       +\rho(1-\overline{w}_i)\|\overline{x}_{\!_{\mathcal {J}_i}}\|\big]\nonumber\\
  & =\nu f(\overline{x})+{\textstyle\sum_{i=1}^m}\phi(\overline{w}_i).\nonumber
 \end{align}
 Notice that $(\overline{x},\overline{w})\in \mathcal{S}$. From the last inequality
 and the arbitrariness of $(x,w)$ in $\mathcal{S}$, it follows that $(\overline{x},\overline{w})\in \mathcal{S}^*$.
 Thus, by the arbitrariness of $(\overline{x},\overline{w})$ in $\mathcal{S}_{\rho}^*$,
 we obtain that $\mathcal{S}_{\rho}^*\subseteq\mathcal{S}^*$.
 Together with $\mathcal{S}^*\subseteq\mathcal{S}_{\rho}^*$, we complete the proof of theorem.
 \end{aproof}

 \bigskip
 \noindent
 {\bf\large Appendix C.}

 \medskip

 To achieve the results of Theorem \ref{errbound1} and Theorem \ref{errbound2},
 we need to establish the following two lemmas where $\delta^{k}\!:=x^{k}\!-\overline{x}$
 and $v^k=e-w^k$ for $k\ge 1$. The first one states a relation between
 $\sum_{i\in (S^{k-1})^c}\!\|\delta^k_{\!_{\mathcal{J}_i}}\|$
 and $\sum_{i\in S^{k-1}}\!\|\delta^k_{\!_{\mathcal{J}_i}}\|$ where $S^{k-1}\supset\overline{S}$
 is an index set.
%-------------------------------------------------------------------------------
 \begin{alemma}\label{noise1-errbound}
  For $k\ge 1$, if there is an index set $S^{k-1}\supseteq \overline{S}$
  such that $\min_{i\in(S^{k-1})^c}w^{k-1}_i\!\le\overline{t}_{\phi}$,
  then with $\lambda^{k-1}\ge \frac{(3-\overline{t}_{\phi})\|\mathcal{G}(\widehat{\varepsilon})\|_\infty}{1-\overline{t}_{\phi}}$
  it holds that
  \(
   \sum_{i\in(S^{k-1})^c}\|\delta^k_{\!_{\mathcal{J}_i}}\|
   \le\frac{2}{1-\overline{t}_{\phi}}\sum_{i\in S^{k-1}}\!\|\delta^k_{\!_{\mathcal{J}_i}}\|.
  \)
 \end{alemma}
 \begin{proof}
 By the optimality of $x^k$ and the feasibility of $\overline{x}$ to the subproblem \eqref{expm-subx},
 we have
 \[
  \frac{1}{2n}\big\|Ax^k-b\big\|^2 +\lambda^{k-1}\sum_{i=1}^mv_i^{k-1}\big\|x^k_{\!_{\mathcal{J}_i}}\big\|
  \le\frac{1}{2n}\big\|A\overline{x}-b\big\|^2 +\lambda^{k-1}\sum_{i=1}^mv_i^{k-1}\big\|\overline{x}_{\!_{\mathcal{J}_i}}\big\|
 \]
 which, by using $\delta^k=x^k-\overline{x}$, $\varepsilon=b-A\overline{x}$
 and $\widehat{\varepsilon}=\frac{1}{n}A^\mathbb{T}\varepsilon$, can be rearranged as follows:
 \[
   \frac{1}{2n}\big\|A\delta^k\big\|^2 \le \langle \widehat{\varepsilon},\delta^k\rangle +\lambda^{k-1}\sum_{i=1}^mv_i^{k-1}\!\left(\big\|\overline{x}_{\!_{\mathcal{J}_i}}\big\|-\big\|x^k_{\!_{\mathcal{J}_i}}\big\|\right).
 \]
 Together with $\overline{x}_{\!_{\mathcal{J}_i}}=0$ for all $i\in\overline{S}^{c}$
 and the definition of $\mathcal{G}(\cdot)$, we obtain that
 \begin{align}\label{noise1p-eb}
  \frac{1}{2n}\big\|A\delta^k\big\|^2
  &\le \sum_{i=1}^m\langle\widehat{\varepsilon}_{\!_{\mathcal{J}_i}},\delta^k_{\!_{\mathcal{J}_i}}\rangle
  +\lambda^{k-1}\sum_{i\in \overline{S}} v_i^{k-1}\left(\big\|\overline{x}_{\!_{\mathcal{J}_i}}\big\|-\big\|x^k_{\!_{\mathcal{J}_i}}\big\|\right)
  -\lambda^{k-1}\sum_{i\in\overline{S}^{c}} v_i^{k-1}\big\|x^k_{\!_{\mathcal{J}_i}}\big\|\nonumber\\
  &\le \langle \mathcal{G}(\widehat{\varepsilon}),\mathcal{G}(\delta^k)\rangle
      +\lambda^{k-1}\sum_{i\in \overline{S}} v^{k-1}_i\big\|\delta^k_{\!_{\mathcal{J}_i}}\big\|
      -\lambda^{k-1}\!\sum_{i\in(S^{k-1})^c}\!v^{k-1}_i\big\|\delta^k_{\!_{\mathcal{J}_i}}\big\|\\
  &\leq \sum_{i\in S^{k-1}\backslash\overline{S}}\big\|\widehat{\varepsilon}_{\!_{\mathcal{J}_i}}\big\|
         \big\|\delta^k_{\!_{\mathcal{J}_i}}\big\| +\big(\lambda^{k-1}+\|\mathcal{G}(\widehat{\varepsilon})\|_\infty\big)
         \sum_{i\in \overline{S}}\big\|\delta^k_{\!_{\mathcal{J}_i}}\big\|\nonumber\\
  &\quad\ +\big[\|\mathcal{G}(\widehat{\varepsilon})\|_\infty-\lambda^{k-1}(1-\overline{t}_{\phi})\big]\!
          {\textstyle\sum_{i\in(S^{k-1})^c}}\,\big\|\delta^k_{\!_{\mathcal{J}_i}}\big\|\nonumber
 \end{align}
 where the last inequality are due to $\min_{i\in(S^{k-1})^{c}}v^{k-1}_i\geq 1-\overline{t}_{\phi}$.
 Then, it holds that
 \begin{align*}
  &\frac{1}{2n}\|A\delta^k\|^2+\Big[\lambda^{k-1}(1-\overline{t}_{\phi})-\|\mathcal{G}(\widehat{\varepsilon})\|_\infty\Big]
   \!\sum_{i\in(S^{k-1})^c}\big\|\delta^k_{\!_{\mathcal{J}_i}}\big\|\nonumber\\
  &\le\sum_{i\in S^{k-1}\backslash\overline{S}}\|\widehat{\varepsilon}_{\!_{\mathcal{J}_i}}\|\|\delta^k_{\!_{\mathcal{J}_i}}\| +(\lambda^{k-1}\!+\|\mathcal{G}(\widehat{\varepsilon})\|_\infty)\sum_{i\in \overline{S}} \|\delta^k_{\!_{\mathcal{J}_i}}\|
%  &\leq\sum_{i\in S^{k-1}\backslash\overline{S}}\|\mathcal{G}(\widehat{\varepsilon})\|_\infty\|\delta^k_{\!_{\mathcal{J}_i}}\| +(\lambda+\|\mathcal{G}(\widehat{\varepsilon})\|_\infty)\sum_{i\in \overline{S}} \|\delta^k_{\!_{\mathcal{J}_i}}\|\\
  \le (\lambda^{k-1}\!+\!\|\mathcal{G}(\widehat{\varepsilon})\|_\infty)\!\sum_{i\in S^{k-1}}\!\|\delta^k_{\!_{\mathcal{J}_i}}\|\nonumber
 \end{align*}
 Together with $\frac{1}{2n}\|A\delta^k\|^2\ge 0$ and $\lambda^{k-1}\geq \frac{(3-\overline{t}_{\phi})\|\mathcal{G}(\widehat{\varepsilon})\|_\infty}{1-\overline{t}_{\phi}}$,
 we get the desired inequality.
 \end{proof}

 When $S^{k-1}$ in Lemma \ref{noise1-errbound} is also such that
 the matrix $A$ satisfies the RSC in $\mathcal {C}(\overline{S},|S^{k-1}|)$ with
 constant $\gamma_k>0$, the result of Lemma \ref{noise1-errbound}
 can be strengthened as follows.
%-------------------------------------------------------------------------
 \begin{alemma}\label{noise2-errbound}
  For $k\ge 1$, if there is an index set $S^{k-1}\supseteq \overline{S}$
  such that $\min_{i\in(S^{k-1})^c}w^{k-1}_i\!\le\overline{t}_{\phi}$
  and $A$ has the RSC over $\mathcal {C}(\overline{S},|S^{k-1}|)$ with constant $\gamma_{k}>0$,
  then with $\lambda^{k-1}\ge \frac{(3-\overline{t}_{\phi})\|\mathcal{G}(\widehat{\varepsilon})\|_\infty}{1-\overline{t}_{\phi}}$
  \[
    \|\delta^k\|\le\frac{1}{\gamma_{k}}\Big(\big\|\left[\mathcal{G}(\widehat{\varepsilon})\right]_{S^{k-1}}\big\| +\lambda^{k-1}\sqrt{{\textstyle\sum_{i\in \overline{S}}}\,(v^{k-1}_i)^2}\Big).
  \]
 \end{alemma}
 \begin{proof}
  Using inequality \eqref{noise1p-eb} and noting that
  $\delta^k\in \mathcal {C}(\overline{S},|S^{k-1}|)$
  by Lemma \ref{noise1-errbound}, we have
  \begin{align}
  \gamma_{k}\|\delta^k\|^2\le\frac{1}{2n}\|A\delta^k\|^2
  &\le \sum_{i=1}^m\|\widehat{\varepsilon}_{\!_{\mathcal{J}_i}}\|\|\delta^k_{\!_{\mathcal{J}_i}}\|
       +\lambda^{k-1}\!\sum_{i\in \overline{S}} v^{k-1}_i\|\delta^k_{\!_{\mathcal{J}_i}}\|
       -\lambda^{k-1}\!\sum_{i\notin S^{k-1}} v^{k-1}_i\|\delta^k_{\!_{\mathcal{J}_i}}\|\nonumber\\
  &\le \sum_{i=1}^m\| \widehat{\varepsilon}_{\!_{\mathcal{J}_i}}\|\|\delta^k_{\!_{\mathcal{J}_i}}\|
       +\lambda^{k-1}\!\sum_{i\in \overline{S}} v^{k-1}_i\|\delta^k_{\!_{\mathcal{J}_i}}\|
       -\lambda^{k-1}(1-\overline{t}_{\phi})\sum_{i\notin S^{k-1}}\|\delta^k_{\!_{\mathcal{J}_i}}\|\nonumber\\
  &\le {\textstyle\sum_{i\in S^{k-1}}}\| \widehat{\varepsilon}_{\!_{\mathcal{J}_i}}\|\|\delta^k_{\!_{\mathcal{J}_i}}\|
        +\lambda^{k-1}{\textstyle\sum_{i\in \overline{S}}}\,v^{k-1}_i\|\delta^k_{\!_{\mathcal{J}_i}}\|\nonumber\\
  &\le \sqrt{{\textstyle\sum_{i\in S^{k-1}}}\| \widehat{\varepsilon}_{\!_{\mathcal{J}_i}}\|^2}\,\big\|\delta^k\big\|
       +\lambda^{k-1}\sqrt{{\textstyle\sum_{i\in \overline{S}}}\,(v^{k-1}_i)^2}\,\big\|\delta^k\big\|\nonumber
 \end{align}
 where the third inequality is by
 $\lambda^{k-1}\ge \frac{(3-\overline{t}_{\phi})\|\mathcal{G}(\widehat{\varepsilon})\|_\infty}{1-\overline{t}_{\phi}}$.
 This implies the desired result.
 \end{proof}

  \noindent
  {\bf\large The proof of Theorem \ref{errbound1}:}
  \begin{aproof}
  For each $k\in\mathbb{N}$, define $S^{k-1}\!:=\overline{S}\cup\{i\notin \overline{S}\!: w_i^{k-1}>\overline{t}_{\phi}\}$.
  Notice that $\nu\le\frac{1-\overline{t}_{\phi}}{(3-\overline{t}_{\phi})\|\mathcal{G}(\widehat{\varepsilon})\|_\infty}$
  and $\frac{(3-\overline{t}_{\phi})\|\mathcal{G}(\widehat{\varepsilon})\|_\infty}{1-\overline{t}_{\phi}}\le\rho\nu^{-1}$.
  We have
  \(
   \lambda^{k-1}\ge\frac{(3-\overline{t}_{\phi})\|\mathcal{G}(\widehat{\varepsilon})\|_\infty}{1-\overline{t}_{\phi}}
  \)
  for all $k\in\mathbb{N}$. If $|S^{k-1}|\leq 1.5\overline{r}$ for some $k\in\mathbb{N}$,
  from Lemma \ref{noise2-errbound} it follows that
  \begin{align}\label{temp-ineq41}
   \|x^k-\overline{x}\|
   &\le \frac{1}{\kappa}\Big(\big\|[\mathcal{G}(\widehat{\varepsilon})]_{S^{k-1}}\big\|
        +\lambda^{k-1}\sqrt{{\textstyle\sum_{i\in \overline{S}}}\,(v^{k-1}_i)^2}\Big)\nonumber\\
   &\leq\frac{1}{\kappa}\Big(\big\|\mathcal{G}(\widehat{\varepsilon})\big\|_\infty\sqrt{1.5\overline{r}}
         +\lambda^{k-1}\sqrt{\overline{r}}\Big)
   \le \frac{\lambda^{k-1}(4-2\overline{t}_{\phi})}{\kappa(3-\overline{t}_{\phi})}\sqrt{1.5\overline{r}},
 \end{align}
  where the last inequality is due to $\|\mathcal{G}(\widehat{\varepsilon})\|_\infty\le
  \frac{1-\overline{t}_{\phi}}{3-\overline{t}_{\phi}}\lambda^{k-1}$ for all $k\in \mathbb{N}$.
  So, it suffices to argue that $|S^{k-1}|\leq 1.5\overline{r}$ for all $k\in\mathbb{N}$.
  When $k=1$, it automatically holds since $S^0=\overline{S}$ by $w^0\le \overline{t}_{\phi}e$.
  Now assume that $|S^{k-1}|\leq 1.5\overline{r}$ for all $k=l$ with $l\geq 1$.
  We shall prove that $|S^{l}|\leq 1.5\overline{r}$.
  Using \eqref{temp-ineq41} with $k=l$, we have $\|x^l-\overline{x}\|\le\frac{\lambda^{l-1}(4-2\overline{t}_{\phi})\sqrt{1.5\overline{r}}}{\kappa(3-\overline{t}_{\phi})}$.
  Notice that $i\in S^l\backslash\overline{S}$ implies $i\notin \overline{S}$ and $w_i^{l}\in(\overline{t}_{\phi},1]$.
  From $w_i^l\in\!\partial\psi^*(\rho\|x_{\!_{\mathcal{J}_i}}^l\|)
  =(\partial\psi)^{-1}(\rho\|x_{\!_{\mathcal{J}_i}}^l\|)$,
  \[
   \rho\|x^{l}_{\!_{\mathcal{J}_i}}\|\ge\psi'_-(w_i^l)=\phi'_-(w_i^l)\geq \phi'_+(\overline{t}_{\phi})\ge\frac{1}{1-t_{\phi}^*},
  \]
  where the equality is due to $\psi_{-}'(t)=\phi'(t)$ for all $t\in(0,1]$.
  This inequality implies
 \begin{align}\label{Lemma-noise42-eb}
  \sqrt{|S^l\backslash\overline{S}|}
  &\le\sqrt{{\textstyle\sum_{i\in S^l\backslash\overline{S}}}\,\rho^2(1-t_{\phi}^*)^2\|x^{l}_{\!_{\mathcal{J}_i}}\|^2}
   \leq \rho(1-t_{\phi}^*)\|x^{l}-\overline{x}\|\\
  &\le\frac{\rho\lambda^{l-1}(1-t_{\phi}^*)(4-2\overline{t}_{\phi})}{\kappa(3-\overline{t}_{\phi})}\sqrt{1.5\overline{r}}
  \le \sqrt{0.5\overline{r}},\nonumber
 \end{align}
 where the last inequality is due to
 $\rho\lambda^{l-1}\le\!\frac{(3-\overline{t}_{\phi})\kappa}{\sqrt{3}(4-2\overline{t}_{\phi})(1-t_{\phi}^*)}$
 implied by $\lambda^{l-1}\!=\nu^{-1}$ for $l=1$ and $\lambda^{l-1}=\rho\nu^{-1}$ for $l>1$.
 So, $|S^{l}|\leq 1.5\overline{r}$.
 Thus, $|S^{k-1}|\leq 1.5\overline{r}$ holds for all $k\in\mathbb{N}$.
 \end{aproof}

 \medskip

 In the following, we upper bound $(v^{k-1}_i)^2$ for $i\in\overline{S}$
 by means of $\mathbb{I}_{\Delta}(i)$ and $\mathbb{I}_{F^k}(i)$.
  %-------------------------------------------------------------------------------------
 \begin{alemma}\label{indexk}
  For each $k\ge 1$, let $F^k$ be the index set defined as in \eqref{Fk}.
  Then, it holds that
  \[
   \sqrt{{\textstyle\sum_{i\in \overline{S}}}(v^{k}_i)^2}
   \leq \sqrt{{\textstyle\sum_{i\in\overline{S}}}\,\mathbb{I}_{\Delta}(i)}
   +\sqrt{{\textstyle\sum_{i\in\overline{S}}}\,\mathbb{I}_{F^k}(i)}.
  \]
 \end{alemma}
 \begin{proof}
  Notice that $v^{k}_i=1-w_i^k\le 1$. If $i\in F^k$, clearly,
  $v^{k}_i\le \mathbb{I}_{F^k}(i)$. Otherwise, together with
  Remark \ref{remark-alg}(a) and $v_i^k=1-w_i^k$, it follows that
  \[
   v_i^k\le\mathbb{I}_{\big\{i:\,\|x^{k}_{\!_{\mathcal{J}_i}}\|\leq \phi'_-(1)/\rho\big\}}(i)
   \le \mathbb{I}_{\big\{i:\,\|\overline{x}_{\!_{\mathcal{J}_i}}\|\le\frac{1}{\rho(1-t^*_{\phi})}+\frac{\phi'_-(1)}{\rho}\big\}}(i)
   \le \mathbb{I}_{\Delta}(i).
  \]
  Hence, for each $i$, it holds that
  \(
  0\leq v^{k}_i\le\mathbb{I}_{\Delta}(i)+\mathbb{I}_{F^k}(i).
  \)
 The desired result follows by noting that $\|a+b\|\le \|a\|+\|b\|$ for all vectors $a$ and $b$.
 \end{proof}

 \medskip
  \noindent
  {\bf\large The proof of Theorem \ref{errbound2}:}
  \begin{aproof}
  For each $k\in\mathbb{N}$, define
 \(
   S^{k-1}:=\overline{S}\cup\{i\notin \overline{S}:\,w_i^{k-1}>\overline{t}_{\phi}\}.
  \)
  Since the conclusion holds for $k=1$, it suffices to consider $k\geq 2$.
  Now, from \eqref{Lemma-noise42-eb},
  \begin{align}\label{Geps-Sk}
   \big\|[\mathcal{G}(\widehat{\varepsilon})]_{S^{k-1}}\big\|
   &\le \big\|[\mathcal{G}(\widehat{\varepsilon})]_{\overline{S}}\big\|
        +\big\|\mathcal{G}(\widehat{\varepsilon})\big\|_\infty\sqrt{|S^{k-1}\backslash\overline{S}|}
    \le\|[\mathcal{G}(\widehat{\varepsilon})]_{\overline{S}}\|
       +\frac{\lambda^{k-1}(1-\overline{t}_{\phi})}{3-\overline{t}_{\phi}}\sqrt{|S^{k-1}\backslash\overline{S}|}\nonumber\\
   &\le \big\|[\mathcal{G}(\widehat{\varepsilon})]_{\overline{S}}\big\|
        +\rho(1-t_{\phi}^*)\lambda^{k-1}\frac{1-\overline{t}_{\phi}}{3-\overline{t}_{\phi}}\big\|x^{k-1}-\overline{x}\big\|.
 \end{align}
  In addition, from Lemma \ref{noise2-errbound} and Lemma \ref{indexk},
  it follows that
  \begin{align*}
   &\|x^k-\overline{x}\|\leq\frac{1}{\kappa}\Big(\|[\mathcal{G}(\widehat{\varepsilon})]_{S^{k-1}}\|
                     +\lambda^{k-1}\sqrt{{\textstyle\sum_{i\in \overline{S}}}\,(v^{k-1}_i)^2}\Big)\nonumber\\
   &\le \frac{1}{\kappa}\Big(\|[\mathcal{G}(\widehat{\varepsilon})]_{S^{k-1}}\|
         +\lambda^{k-1}\sqrt{{\textstyle\sum_{i\in \overline{S}}}\,\mathbb{I}_{\Delta}(i)}
         +\lambda^{k-1}\sqrt{{\textstyle\sum_{i\in \overline{S}}}\,\mathbb{I}_{F^k}(i)}\Big)\nonumber\\
   &\le \frac{1}{\kappa}\Big(\|[\mathcal{G}(\widehat{\varepsilon})]_{S^{k-1}}\|
       +\lambda^{k-1}\sqrt{{\textstyle\sum_{i\in \overline{S}}}\,\mathbb{I}_{\Delta}(i)}
       +\lambda^{k-1}\sqrt{{\textstyle\sum_{i\in \overline{S}}}\big[(1\!-\!t_{\phi}^*)^2
       |\|x^{k-1}_{\!_{\mathcal{J}_i}}\|-\|\overline{x}_{\!_{\mathcal{J}_i}}\||^2\rho^2\big]}\Big)\nonumber\\
   &\le \frac{1}{\kappa}\Big(\|[\mathcal{G}(\widehat{\varepsilon})]_{S^{k-1}}\|
        +\rho(1-t_{\phi}^*)\lambda^{k-1}\|x^{k-1}-\overline{x}\|+
       \lambda^{k-1}\sqrt{{\textstyle\sum_{i\in \overline{S}}}\,\mathbb{I}_{\Delta}(i)}\Big).
  \end{align*}
  where the third inequality is by the definition of $F^k$.
  Together with \eqref{Geps-Sk}, we obtain
  \begin{align*}
   \|x^k-\overline{x}\|
   &\le \frac{1}{\kappa}\Big(\|[\mathcal{G}(\widehat{\varepsilon})]_{\overline{S}}\|
        +\frac{\rho\lambda^{k-1}(1-t_{\phi}^*)(4-2\overline{t}_{\phi})\|x^{k-1}-\overline{x}\|}{3-\overline{t}_{\phi}}
        +\lambda^{k-1}\sqrt{{\textstyle\sum_{i\in \overline{S}}}\,\mathbb{I}_{\Lambda}(i)}\Big)\\
   &\le \frac{1}{\kappa}\Big(\|[\mathcal{G}(\widehat{\varepsilon})]_{\overline{S}}\|
                +\lambda^{k-1}\sqrt{{\textstyle\sum_{i\in \overline{S}}}\,\mathbb{I}_{\Delta}(i)}\Big)
                +\frac{1}{\sqrt{3}}\|\delta^{k-1}\|\\
   &=\frac{1}{\kappa}\Big(\|[\mathcal{G}(\widehat{\varepsilon})]_{\overline{S}}\|
                +\rho\nu^{-1}\sqrt{{\textstyle\sum_{i\in \overline{S}}}\,\mathbb{I}_{\Delta}(i)}\Big)
                +\frac{1}{\sqrt{3}}\|x^{k-1}-\overline{x}\|
 \end{align*}
  where the second inequality is using $\rho\lambda^{k-1}\le\frac{(3-\overline{t}_{\phi})\kappa}{\sqrt{3}(4-2\overline{t}_{\phi})(1-t_{\phi}^*)}$,
  and the last one is using $\lambda^{k-1}=\rho\nu^{-1}$ for $k\ge 2$.
  The desired result follows by this recursion inequality.
 \end{aproof}

 In order to achieve the result of Theorem \ref{theorem-consistency},
 we also need the following two technical lemmas
 where $\widehat{\delta}^{k}\!:=x^{k}\!-x^{{\rm LS}}$ for $k=1,2,\ldots$.

%--------------------------------------------------------------------------------------------------
 \begin{alemma}\label{Lemma-noise1}
  For $k\ge 1$, if there is an index set $S^{k-1}\supseteq \overline{S}$
  such that $\min_{i\in(S^{k-1})^c}w^{k-1}_i\!\le\overline{t}_{\phi}$,
  then with $\lambda^{k-1}\geq \frac{2\|\mathcal{G}(\varepsilon^{{\rm LS}})\|_\infty}{1-\overline{t}_{\phi}}$
  it holds that
  \(
    \sum_{i\in (S^{k-1)^c}}\|\widehat{\delta}^k_{\!_{\mathcal{J}_i}}\|
    \le\frac{2}{1-\overline{t}_{\phi}}\sum_{i\in S^{k-1}}\|\widehat{\delta}^k_{\!_{\mathcal{J}_i}}\|.
  \)
 \end{alemma}
 \begin{proof}
 By the optimality of $x^k$ and the feasibility of $x^{{\rm LS}}$ to the subproblem \eqref{expm-subx},
  \[
   \frac{1}{2n}\|Ax^k-b\|^2 +\lambda^{k-1}\sum_{i=1}^mv_i^{k-1}\|x^k_{\!_{\mathcal{J}_i}}\|
   \leq\frac{1}{2n}\|Ax^{{\rm LS}}-b\|^2 +\lambda^{k-1}\sum_{i=1}^mv_i^{k-1}\|x^{{\rm LS}}_{\!_{\mathcal{J}_i}}\|,
  \]
 which, by the definition of $\varepsilon^{{\rm LS}}$, can be rearranged as follows:
  \[
   \frac{1}{2n}\|A\widehat{\delta}^k\|^2
   \le -\langle \varepsilon^{{\rm LS}},\widehat{\delta}^k\rangle
       +\lambda^{k-1}\sum_{i=1}^mv_i^{k-1}(\|x^{{\rm LS}}_{\!_{\mathcal{J}_i}}\|-\|x^k_{\!_{\mathcal{J}_i}}\|).
  \]
  Together with $x^{{\rm LS}}_{\!_{\mathcal{J}_i}}=0$ for all $i\notin\overline{S}$,
  it immediately follows that
  \begin{align}
   \frac{1}{2n}\|A\widehat{\delta}^k\|^2
   &\le-\sum_{i=1}^m\langle \varepsilon^{{\rm LS}}_{\!_{\mathcal{J}_i}},\widehat{\delta}^k_{\!_{\mathcal{J}_i}}\rangle
        +\lambda^{k-1}\sum_{i\in \overline{S}} v_i^{k-1}(\|x^{{\rm LS}}_{\!_{\mathcal{J}_i}}\|-\|x^k_{\!_{\mathcal{J}_i}}\|)
        -\lambda^{k-1}\sum_{i\notin \overline{S}} v_i^{k-1}\|x^k_{\!_{\mathcal{J}_i}}\|\nonumber\\
   &\le \sum_{i\notin \overline{S}}\|\varepsilon^{{\rm LS}}_{\!_{\mathcal{J}_i}}\|\|\widehat{\delta}^k_{\!_{\mathcal{J}_i}}\|
      +\lambda^{k-1}\sum_{i\in \overline{S}} v^{k-1}_i\|\widehat{\delta}^k_{\!_{\mathcal{J}_i}}\|
      -\lambda^{k-1}\sum_{i\notin S^{k-1}} v^{k-1}_i\|\widehat{\delta}^k_{\!_{\mathcal{J}_i}}\|\nonumber\\
   &\le \sum_{i\notin \overline{S}}\| \varepsilon^{{\rm LS}}_{\!_{\mathcal{J}_i}}\|\|\widehat{\delta}^k_{\!_{\mathcal{J}_i}}\|
       +\lambda^{k-1}\sum_{i\in \overline{S}} v^{k-1}_i\big\|\widehat{\delta}^k_{\!_{\mathcal{J}_i}}\big\|
       -\lambda^{k-1}(1-\overline{t}_{\phi})\sum_{i\notin S^{k-1}}\!\big\|\widehat{\delta}^k_{\!_{\mathcal{J}_i}}\big\|,\nonumber
  \end{align}
  where the equality is due to $\varepsilon^{{\rm LS}}_{\mathcal {J}_{\overline{S}}}=0$ implied by
  the optimality of $x^{{\rm LS}}$ to \eqref{defxls}. Thus,
  \begin{align}\label{Lemma-noise1p}
   &\frac{1}{2n}\|A\widehat{\delta}^k\|^2+\big[\lambda^{k-1}(1-\overline{t}_{\phi})-\|\mathcal{G}(\varepsilon^{{\rm LS}})\|_\infty\big]
    \sum_{i\notin S^{k-1}}\big\|\widehat{\delta}^k_{\!_{\mathcal{J}_i}}\big\| \nonumber\\
   &\le \sum_{i\in S^{k-1}\backslash\overline{S}}\big\|\varepsilon^{{\rm LS}}_{\!_{\mathcal{J}_i}}\big\|
        \big\|\widehat{\delta}^k_{\!_{\mathcal{J}_i}}\big\|
        +\lambda^{k-1}\sum_{i\in \overline{S}} v^{k-1}_i\|\widehat{\delta}^k_{\!_{\mathcal{J}_i}}\|
   \le \lambda^{k-1}\sum_{i\in S^{k-1}} \|\widehat{\delta}^k_{\!_{\mathcal{J}_i}}\|.
 \end{align}
  Notice that $\lambda^{k-1}\geq \frac{2\|\mathcal{G}(\varepsilon^{LS})\|_\infty}{1-\overline{t}_{\phi}}$.
  The desired result directly follows from \eqref{Lemma-noise1p}.
 \end{proof}

 If in addition the index set $S^{k-1}$ in Lemma \ref{Lemma-noise1} is such that $A$ satisfies
 the RSC over $\mathcal{C}(\overline{S},|S^{k-1}|)$, then the conclusion of Lemma \ref{Lemma-noise1}
 can be strengthened as follows.
%-----------------------------------------------------------------------------------------
 \begin{alemma}\label{Lemma-noise2}
  For $k\ge 1$, if there is an index set $S^{k-1}\supseteq \overline{S}$
  such that $\min_{i\in(S^{k-1})^c}w^{k-1}_i\!\le\overline{t}_{\phi}$
  and $A$ satisfies the RSC over $\mathcal {C}(\overline{S},|S^{k-1}|)$ with constant $\gamma_{k}$,
  then with $\lambda^{k-1}\geq \frac{2\|\mathcal{G}(\varepsilon^{LS})\|_\infty}{1-\overline{t}_{\phi}}$
  \[
   \|\widehat{\delta}^k\|\le\frac{1}{\gamma_{k}}\Big(\|\mathcal{G}(\varepsilon^{\rm{LS}})_{S^{k-1}}\|
   +\lambda^{k-1}\sqrt{{\textstyle\sum_{i\in \overline{S}}}(v^{k-1}_i)^2}\Big).
 \]
 \end{alemma}
 \begin{proof}
  By using the first inequality in \eqref{Lemma-noise1p} and $\widehat{\delta}^k\in\mathcal {C}(|S^{k-1}|)$
  by Lemma \ref{Lemma-noise1}, we have
  \begin{align}
   \gamma_{k}\|\widehat{\delta}^k\|^2 \le\frac{1}{2n}\|A\widehat{\delta}^k\|^2
   &\le \sum_{i\in S^{k-1}\backslash\overline{S}}\!\|\varepsilon^{\rm{LS}}_{\!_{\mathcal{J}_i}}\|\|\widehat{\delta}^k_{\!_{\mathcal{J}_i}}\|
     +\lambda^{k-1}\sum_{i\in \overline{S}} v^{k-1}_i\|\widehat{\delta}^k_{\!_{\mathcal{J}_i}}\|\nonumber\\
   &\le \sqrt{{\textstyle\sum_{i\in S^{k-1}\backslash\overline{S}}}\,
    \|\varepsilon^{\rm{LS}}_{\!_{\mathcal{J}_i}}\|^2}\,\big\|\widehat{\delta}^k\big\|
   +\lambda^{k-1}\sqrt{{\textstyle\sum_{i\in \overline{S}}}\,(v^{k-1}_i)^2}\,\big\|\widehat{\delta}^k\big\|,
  \end{align}
  where the last inequality is using $\|\widehat{\delta}^k\|^2=\sum_{i=1}^m\|\widehat{\delta}^k_{\!_{\mathcal{J}_i}}\|^2$.
  Thus, it follows that
  \begin{align*}
    \gamma_{k}\|\widehat{\delta}^k\|
    &\le\!\sqrt{{\textstyle\sum_{i\in S^{k-1}\backslash\overline{S}}}\,
    \|\varepsilon^{\rm{LS}}_{\!_{\mathcal{J}_i}}\|^2}
       +\lambda^{k-1}\sqrt{{\textstyle\sum_{i\in \overline{S}}}\,(v^{k-1}_i)^2}\\
    &=\|[\mathcal{G}(\varepsilon^{\rm{LS}})]_{S^{k-1}}\|
       +\!\lambda^{k-1}\sqrt{{\textstyle\sum_{i\in \overline{S}}}\,(v^{k-1}_i)^2},
  \end{align*}
  which implies the desired result. The proof is then completed.
  \end{proof}

    \medskip
  \noindent
  {\bf\large The proof of Theorem \ref{theorem-consistency}:}
  \begin{aproof}
  From $\nu\le\frac{1-\overline{t}_{\phi}}{2\|\mathcal{G}(\varepsilon^{\rm{LS}})\|_\infty}$
  and $\frac{\rho}{\nu}\ge \frac{2\max(\|\mathcal{G}(\varepsilon^{\rm{LS}})\|_\infty,
  2\kappa\|\mathcal{G}(\widehat{\varepsilon}^\dag)\|_\infty)}{1-\overline{t}_{\phi}}$,
  $\lambda^{k-1}\geq \frac{2\|\mathcal{G}(\varepsilon^{LS})\|_\infty}{1-\overline{t}_{\phi}}$
  for all $k\in \mathbb{N}$. We prove that the inequalities in \eqref{errk}
  hold for $k=1$. Since $w^0\le\overline{t}_{\phi}e$ and $S^0=\overline{S}$,
  the conditions of Lemma \ref{Lemma-noise2} are satisfied for $k=1$. Then,
 \begin{align*}
  \|x^1-x^{\rm{LS}}\|
  \le\frac{1}{\kappa}\Big(\|[\mathcal {G}(\varepsilon^{\rm{LS}})]_{S^{0}}\|
     +\!\lambda^0\sqrt{{\textstyle\sum_{i\in \overline{S}}}(v^{0}_i)^2}\Big)
  \le \frac{1}{\kappa}\Big(\|[\mathcal {G}(\varepsilon^{\rm{LS}})]_{\overline{S}}\|+\!\frac{\sqrt{|F^0|}}{\nu}\Big)
  =\frac{\sqrt{|F^0|}}{\kappa\nu}.
 \end{align*}
  where the equality is due to $[\mathcal {G}(\varepsilon^{\rm{LS}})]_{\overline{S}}=0$.
  Since $\|x^{\rm{LS}}_{\!_{\mathcal {J}_i}}-\overline{x}_{\!_{\mathcal {J}_i}}\|
  \le \|\mathcal{G}(\widehat{\varepsilon}^\dagger)\|_\infty$ for $i=1,\ldots,m$
  by \eqref{xLS-err} and $\frac{\rho}{\nu}\ge\frac{4\kappa\|\mathcal{G}(\widehat{\varepsilon}^\dag)\|_\infty}{1-\overline{t}_{\phi}}$,
  it follows that for all $i\in F^{1}$,
  \[
  \|x^{\rm{LS}}_{\!_{\mathcal {J}_i}}\!-x^1_{\!_{\mathcal {J}_i}}\|
   \ge \|\overline{x}_{\!_{\mathcal {J}_i}}\!-\!x^1_{\!_{\mathcal {J}_i}}\|
      -\|\overline{x}_{\!_{\mathcal {J}_i}}\!-x^{\rm{LS}}_{\!_{\mathcal {J}_i}}\|\ge\frac{1}{(1\!-t_{\phi}^*)\rho}-\frac{\rho(1-\!\overline{t}_{\phi})}{4\nu\kappa}
   \ge\frac{3\sqrt{5}}{(1\!+3\sqrt{5})(1\!-t_{\phi}^*)\rho},
  \]
  where the last inequality is due to $\rho\le\!\sqrt{\frac{4\kappa\nu}{(1-t_{\phi}^*)(1+3\sqrt{5})}}$.
  From the last two inequalities,
 \begin{align*}
  \sqrt{|F^{1}|}
  &\le \frac{(1+3\sqrt{5})(1-t_{\phi}^*)\rho}{3\sqrt{5}}
    \sqrt{{\textstyle\sum_{i=1}^m}\|x^{\rm{LS}}_{\!_{\mathcal{J}_i}}-x^1_{\!_{\mathcal{J}_i}}\|^2}\\
  &=\frac{\rho(1-t_{\phi}^*)(1+3\sqrt{5})\|x^{\rm{LS}}-x^1\|}{3\sqrt{5}}
  \le\frac{\rho(1\!-\!t_{\phi}^*)(1\!+\!3\sqrt{5})}{6\nu\kappa}\sqrt{|F^0|}.
 \end{align*}
 Consequently, the conclusion holds for $k=1$. Now assuming that the conclusion holds
 for $k=l-1$ with $l\ge 2$, we shall prove that the conclusion holds for $k=l$.
 To this end, we first argue that $|S^{l-1}|\leq 1.5\overline{r}$. Indeed,
 if $i\in S^{l-1}\backslash\overline{S}$, we have $i\notin \overline{S}$
 and $w_i^{l-1}\in(\overline{t}_{\phi},1]$. From $w_i^{l-1}\in\!\partial\psi^*(\rho\|x_{\!_{\mathcal{J}_i}}^{l-1}\|)$,
 we have $\rho\|x^{l-1}_{\!_{\mathcal{J}_i}}\|
 \ge\phi'_-(w_i^{l-1})\geq \phi'_+(\overline{t}_{\phi})\geq\frac{1}{1-t_{\phi}^*}$. Thus,
 \begin{align}\label{induction-ineq}
  \sqrt{|S^{l-1}\backslash\overline{S}|}
  &\le\sqrt{|F^{l-1}|}\le \frac{\rho^2(1-t_{\phi}^*)(1+3\sqrt{5})}{6\nu\kappa}\sqrt{|F^{l-2}|}\leq\cdots\nonumber\\
  &\le\Big(\frac{\rho^2(1-t_{\phi}^*)(1+3\sqrt{5})}{6\nu\kappa}\Big)^{l-1}\sqrt{|F^0|}
  \le\sqrt{(4/6)^{2l-2}|F^{0}|}\leq\sqrt{0.5\overline{r}},
 \end{align}
 where the first inequality is due to $S^{l-1}\backslash\overline{S}\subseteq F^{l-1}$
 and $l\ge 2$. This implies $|S^{l-1}|\leq 1.5\overline{r}$, and hence the conditions of
 Lemma \ref{Lemma-noise2} in Appendix C are satisfied. Consequently,
 \begin{align}
  \|x^l-x^{\rm{LS}}\|
  &\le\frac{1}{\kappa}\Big(\|[\mathcal{G}(\varepsilon^{\rm{LS}})]_{S^{l-1}}\|
            +\lambda^{l-1}\sqrt{{\textstyle\sum_{i\in \overline{S}}}\,(v^{l-1}_i)^2}\Big)\nonumber\\
  &\le\frac{1}{\kappa}\Big(\|\mathcal{G}(\varepsilon^{\rm{LS}})_{S^{l-1}\backslash\overline{S}}\|
      +\lambda^{l-1}\sqrt{{\textstyle\sum_{i\in\overline{S}}}\,\mathbb{I}_{F^{l-1}}(i)}\Big)\nonumber\\
  &\le \frac{1}{\kappa}\Big(\|\mathcal{G}(\varepsilon^{\rm{LS}})\|_\infty\sqrt{|S^{l-1}\backslash\overline{S}|}
       +\lambda^{l-1}\sqrt{|F^{l-1}\cap \overline{S}|}\Big)\nonumber\\
  &\leq\frac{\lambda^{l-1}}{\kappa}\Big(\frac{1}{2}\sqrt{|F^{l-1}\backslash\overline{S}|}+\sqrt{|F^{l-1}\cap \overline{S}|}\Big)
  \leq\frac{\rho}{\nu\kappa}\sqrt{1.25|F^{l-1}|}\leq\frac{\sqrt{5}\rho}{2\nu\kappa}\sqrt{|F^{l-1}|},\nonumber
  \end{align}
  where the second inequality is using $\varepsilon^{\rm{LS}}_{\mathcal {J}_{\overline{S}}}=0$,
  $\rho>\frac{2\phi'_-(1)}{\min_{i\in\overline{S}}\|\overline{x}_{\!_{\mathcal {J}_i}}\|}$
  and Lemma \ref{indexk}, the fourth one is due to $\lambda^{l-1}\geq2\|\mathcal{G}(\varepsilon^{\rm{LS}})\|_\infty$,
  and the fifth one is since $\frac{1}{2}a+b\leq\sqrt{1.25(a^2+b^2)}$ for all $a,b\in \mathbb{R}$.
  In addition, by using the same argument as those for $k=1$, for all $i\in F^{l}$ we have
  \(
   \|x^l_{\!_{\mathcal{J}_i}}-x^{\rm{LS}}_{\!_{\mathcal{J}_i}}\|
   \ge \frac{3\sqrt{5}}{1+3\sqrt{5}}\frac{1}{(1-t_{\phi}^*)\rho},
  \)
  and hence
  \(
   \sqrt{|F^{l}|}\le\frac{\rho^2(1-t_{\phi}^*)(1+3\sqrt{5})}{6\nu\kappa}\sqrt{|F^{l-1}|}.
  \)
  Thus, we complete the proof of the case $k=l$, and the inequalities in \eqref{errk} hold.

  \medskip

  Since $\rho^2\nu^{-1}(1-t_{\phi}^*)(1+3\sqrt{5})\le 4\kappa$,
  we have $\frac{\rho^2(1-t_{\phi}^*)(1+3\sqrt{5})}{6\nu\kappa}\le \frac{2}{3}$.
  Together with \eqref{induction-ineq},
  \[
   \sqrt{|F^{\overline{k}}|}
   \le\Big(\frac{\rho^2(1-t_{\phi}^*)(1+3\sqrt{5})}{6\nu\kappa}\Big)^{\overline{k}-1}
    \sqrt{|F^0|}<1,
  \]
  which implies that $|F^{k}|=0$ when $k\ge \overline{k}$.
  Together with the first inequality in \eqref{errk}, we have
  $x^k=x^{\rm{LS}}$ when $k\geq \overline{k}$.
  From \eqref{xLS-err} and $\rho\le\!\sqrt{\frac{4\kappa\nu}{(1-t_{\phi}^*)(1+3\sqrt{5})}}$,
  for all $i\in \overline{S}$,
  \[
   |\|\overline{x}_{\!_{\mathcal{J}_i}}\|-\|x^{LS}_{\!_{\mathcal{J}_i}}\||
   \le\|\overline{x}_{\!_{\mathcal{J}_i}}-x^{LS}_{\!_{\mathcal{J}_i}}\|
   \leq \|\mathcal{G}(\widehat{\varepsilon}^\dagger)\|_\infty\le\frac{\rho(1-\overline{t}_{\phi})}{4\nu\kappa}
   \le \frac{(1-\overline{t}_{\phi})}{(1\!-t_{\phi}^*)(1+3\sqrt{5})\rho}.
  \]
  This, together with $\min_{i\in\overline{S}}\|\overline{x}_{\!_{\mathcal{J}_i}}\|\ge\frac{2\phi'_-(1)}{\rho}
  \geq\frac{2}{\rho(1-t_{\phi}^*)}$, implies that $\|x^{\rm{LS}}_{\!_{\mathcal{J}_i}}\|>0$ for all $i\in \overline{S}$.
  Thus, ${\rm supp}(\mathcal{G}(x^{\rm{LS}}))=\overline{S}$, and consequently
  ${\rm supp}(\mathcal{G}(x^{k}))=\overline{S}$ for all $k\ge\overline{k}$.
 \end{aproof}

 \end{document}